\documentclass[final,nomarks]{dmtcs-episciences}

\usepackage[utf8]{inputenc}
\usepackage{subfigure}

\author[N. Alshammari and D. Bevan]{Noura Alshammari \and David Bevan}

\title[Enumeration and limit shapes of grid classes]{On the asymptotic enumeration and limit shapes of monotone grid classes of permutations}

\affiliation{Department of Mathematics and Statistics, The University of Strathclyde, Glasgow, Scotland}


\RequirePackage{amsmath,amssymb,amsthm}
\RequirePackage{perms}
\RequirePackage{needspace}
\newtheorem{thm}{Theorem}[section]
\newtheorem{obs}[thm]{Observation}
\newtheorem{prop}[thm]{Proposition}
\newenvironment{smallmx}[1][{}] {\left(\!\begin{smallmatrix}} {\end{smallmatrix}\!\right)}
\newcommand{\+}{\hspace{0.07 em}}
\newcommand{\av}{\mathsf{Av}}
\newcommand{\bbP}{\mathbb{P}}
\newcommand{\bs}{\boldsymbol\sigma}
\newcommand{\bsh}{\bs^\#}
\newcommand{\bshn}{\bs^\#_{\!n}}
\newcommand{\bsn}{\bs_{\!n}}
\newcommand{\CB}{C_{\mathsf{B}}}
\newcommand{\cBlank}{\!\gcone10\!}
\newcommand{\CCC}{\mathcal{C}}
\newcommand{\cDown}{\!\gcone1{-1}\!}
\newcommand{\CR}{C_{\mathsf{R}}}
\newcommand{\cUp}{\!\gcone11\!}
\newcommand{\darkgrey}{\color{gray!50!black}}
\newcommand{\defeq}{\mathchoice{\;:=\;}{:=}{:=}{:=}}
\newcommand{\eq}{\mathchoice{\;=\;}{=}{=}{=}}
\newcommand{\Geom}{\mathsf{Geom}}
\newcommand{\geqs}{\geqslant}
\newcommand{\GGG}{\mathcal{G}}
\newcommand{\gr}{\mathrm{gr}}
\newcommand{\Grid}{\mathsf{Grid}}
\newcommand{\Gridhash}{\Grid^\#}
\newcommand{\HHH}{\mathcal{H}}
\newcommand{\leqs}{\leqslant}
\newcommand{\liminfty}[1][n]{\lim\limits_{#1\rightarrow\infty}}
\newcommand{\longDown}{\!\gctwo2{-1,0}{0,-1}\!}
\newcommand{\longUp}{\!\gctwo2{0,1}{1,0}\!}
\newcommand{\longUpB}{\!\gctwo2{1,1}{1,0}\!}
\newcommand{\prob}[1]{\bbP\big[#1\big]}
\newcommand{\px}{p_{\xxx}}
\newcommand{\py}{p_{\yyy}}
\newcommand{\pz}{p_{\zzz}}
\newcommand{\qx}{q_{\xxx}}
\newcommand{\qy}{q_{\yyy}}
\newcommand{\qz}{q_{\zzz}}
\newcommand{\rot}[1]{#1^{{}_\mathsf{R}}}
\newcommand{\SSS}{\mathcal{S}}
\newcommand{\STAR}{{\scalebox{1.1}{$*$}}}
\newcommand{\tee}{\!\gctwo2{-1,1}{1,0}\!}

\newcommand{\veps}{\varepsilon}
\newcommand{\VVV}{\mathcal{V}}
\newcommand{\xxx}{\mathsf{x}}
\newcommand{\yyy}{\mathsf{y}}
\newcommand{\zzz}{\mathsf{z}}

\begin{document}

\publicationdata{vol. 27:1, Permutation Patterns 2024}{2025}{2}{10.46298/dmtcs.14018}{2024-08-02; 2024-08-02; 2025-05-06; 2025-06-09}{2025-06-19}

\maketitle

\begin{abstract}
We exhibit a procedure to asymptotically enumerate monotone grid classes of permutations.
This is then applied to compute the asymptotic number of permutations in any connected one-corner class.
Our strategy consists of
enumerating the gridded permutations,
finding the asymptotic distribution of points between the cells in a typical large gridded permutation,
and analysing in detail the ways in which a typical permutation can be gridded.
We also determine the limit shape of any connected monotone grid class.
\end{abstract}

\section{Introduction}

We consider a permutation of {length} $n$ (an \emph{$n$-permutation}) to be a linear ordering of 
$1,2,\ldots,n$, and usually identify a permutation $\sigma=\sigma_1\ldots\sigma_n$ with its \emph{plot}, the set of points $(i,\sigma_i)$ in the Euclidean plane.
The monotone \emph{grid class} $\Grid(M)$ is a set of permutations defined by a \emph{gridding matrix} $M$ whose entries are drawn
from $\{\cUp,\cDown,\cBlank\}$.
This matrix specifies the permitted shape for plots of permutations in the class.
Each entry of $M$ corresponds to a \emph{cell} in an \emph{$M$\!-gridding} of a permutation.
If the entry is $\cUp$, any points in the cell must form an increasing sequence;
if the entry is $\cDown$, any points in the cell must form a decreasing sequence;
if the entry is $\cBlank$, then the cell must be empty.
A~permutation can have more than one $M$\!-gridding.
See Figure~\ref{figGriddings} for an example.
Formal definitions are given below.

The study of individual grid classes goes back to the work of Stankova~\cite{Stankova1994} and Atkinson~\cite{Atkinson1998} in the 1990s on the skew-merged permutations.
However, monotone grid classes were formally introduced (under a different name) in 2003 by Murphy and Vatter~\cite{MV2003}.
Huczynska and Vatter~\cite{HV2006} were the first to call them ``grid classes'' and to use the notation $\Grid(M)$.
For a presentation of results concerning grid classes, see Section~12.4 of Vatter's survey article on permutation classes~\cite{VatterSurvey}.

\emph{Exact} enumerative results for monotone grid classes are rare.
A few individual classes have been enumerated, such as the \emph{skew-merged} permutations $\Grid\big(\!\gctwo{2}{-1,1}{1,-1}\!\big)$
whose generating function was determined by Atkinson~\cite{Atkinson1998}.
However, only two \emph{general} results are known.
The first of these concerns \emph{skinny} grid classes, which are those whose matrices have a single row.
That is, $\Grid(M)$ is skinny if $M$ is simply a vector over $\{\!\cUp,\cDown\!\}$.
In the second author's PhD thesis~\cite{BevanThesis},
an iterative procedure is presented which yields the generating function of any skinny monotone grid class, a result which was later generalized by Brignall and Slia\v{c}an in~\cite{BS2019} to skinny classes in which the contents of a single cell may be non-monotone.
One special case, due to Asinowski, Banderier and Hackl~\cite{ABH2021}, resolving a conjecture in~\cite{BevanThesis}, is that
the ordinary generating function for the $k$-celled skinny class
$\Grid({\!\gcone{2}{1,1}\!\cdots\!\gcone1{1}\!})$
is given by the sum
\[
G_k(z) \;=\; \sum_{r=1}^k \, \frac{1}{1-r\+ z} \left( \frac{r\+ z}{r\+ z-1} \right)^{k-r}.
\]
The other general result concerns \emph{polynomial} grid classes (classes with growth rate equal to 1).
The monotone grid class $\Grid(M)$ is polynomial if $M$ has at most one $\cUp$ or $\cDown$ in any column or row.
Homberger and Vatter~\cite{HV2016} describe an algorithm which can enumerate any such class.
Extending either of these general exact enumeration results seems very challenging.
Perhaps the \emph{Combinatorial Exploration} algorithmic framework described in~\cite{ABCNPU2022} can be adapted to do so.
Michal Opler reports\footnote{Personal communications, 2024--2025.} recent work on developing a method to automatically compute generating functions of small acyclic monotone grid classes through the use of monadic second-order logic, based on~\cite{Braunfeld2025}.

Closely related to monotone grid classes are \emph{geometric grid classes}, introduced and studied in~\cite{AABRV2013} and~\cite{BevanGeomGR}.
Given a gridding matrix $M$, the geometric grid class $\Geom(M)$ is the subset of $\Grid(M)$ consisting of those permutations whose points can be plotted on the diagonal lines of the gridding matrix.
We have $\Geom(M)=\Grid(M)$ if and only if the cell graph of $M$ is acyclic (see Section~\ref{sectMaximalMDistributionUnique} for a definition of the cell graph of a gridding matrix).
The monotone grid classes that are the primary concern of this paper are acyclic, so results about these also apply to the corresponding geometric classes.
However, we do not concern ourselves further with geometric grid classes; see~\cite{AABRV2013} for more on their properties.

\begin{figure}[t]
  \centering
  \begin{tikzpicture}[scale=0.225] \plotgriddedperm{9}{8,7,9,6,1,4,2,3,5}{2,6}{1} 
  \end{tikzpicture} ~~~~~~~~
  \begin{tikzpicture}[scale=0.225] \plotgriddedperm{9}{8,7,9,6,1,4,2,3,5}{2,7}{1} 
  \end{tikzpicture} ~~~~~~~~
  \begin{tikzpicture}[scale=0.225] \plotgriddedperm{9}{8,7,9,6,1,4,2,3,5}{2,7}{2} 
  \end{tikzpicture} ~~~~~~~~
  \begin{tikzpicture}[scale=0.225] \plotgriddedperm{9}{8,7,9,6,1,4,2,3,5}{2,8}{2} 
  \end{tikzpicture} ~~~~~~~~
  \begin{tikzpicture}[scale=0.225] \plotgriddedperm{9}{8,7,9,6,1,4,2,3,5}{2,8}{3} 
  \end{tikzpicture}
  \caption{The five \!\gctwo{3}{-1,-1,1}{0,1}\!-griddings of 879614235}\label{figGriddings}
\end{figure}

Given the difficulty of exact enumeration, our goal is the \emph{asymptotic}
enumeration of monotone grid classes.
In~\cite{BevanGridGR}, the second author proves that the exponential \emph{growth rate},
\[
\gr\big(\Grid(M)\big) \defeq \liminfty\big|\Grid_n(M)\big|^{1/n},
\]
 of $\Grid(M)$ always exists and is equal to the square of the spectral radius of a certain graph 
associated with~$M$.
A~simpler proof of a more general version of this result is given by Albert and Vatter~\cite{AV2019}.
Thus we know that
\[
\big|\Grid_n(M)\big| \;\sim\; \theta(n)\,g^n ,
\]
where $g$ is the growth rate of the class and $\theta(n)$ is subexponential; that is, $\liminfty\theta(n)^{1/n}=1$.
We write $f(n)\sim g(n)$ to denote that $\liminfty f(n)/g(n)=1$.

Our primary contribution is a procedure to determine the asymptotic enumeration of a monotone grid class.
This is first applied to skinny classes (Theorem~\ref{thmSkinnyAsympt}).
We then establish the asymptotic distribution of points between cells in a typical large gridded permutation (Theorem~\ref{thmConnectedDistrib}),
and use this to
determine grid class limit shapes (Theorem~\ref{thmLimitShape}).
Building on these results, the major part of our work then concerns the asymptotic enumeration of connected one-corner classes (Theorem~\ref{thmLTXAsympt}).
These are either \textsf{L}-shaped, \textsf{T}-shaped or \textsf{X}-shaped.
See Section~\ref{sectLTX} for the relevant definitions.

We begin with a formal definition of an $M$\!-gridding.\label{defMGridding}
In this definition, in order to match the way we view permutations graphically, we index matrices from the lower left corner, with the order of the indices reversed from the normal convention.
Thus, 
$M_{2,1}$ is the entry that is in the second column from the left and in the bottom row of~$M$.

If $M$ is a gridding matrix with $t$ columns and $u$ rows, then formally an \emph{$M$\!-gridding} of a permutation $\sigma$ of length~$n$
is a pair of integer sequences
$0=c_0\leqs c_1\leqs\ldots\leqs c_t=n$ (the \emph{column dividers})
and $0=r_0\leqs r_1\leqs \ldots\leqs r_u=n$ (the \emph{row dividers})
such that
for every column $i$ and row $j$,
the subsequence of $\sigma$ with indices in $(c_{i-1},c_i]$ and values in $(r_{j-1},r_j]$ is
increasing if $M_{i,j}=\cUp$,
decreasing if $M_{i,j}=\cDown$, and
empty if $M_{i,j}=\cBlank$.
For example, in the leftmost gridding in Figure~\ref{figGriddings}, we have $c_1=2$, $c_2=6$ and $r_1=1$.
If the gridding matrix is clear from the context, we may just talk of a \emph{gridding} of a permutation.

In practice, it is more natural to consider the assignment of points to cells that is induced by an $M$\!-gridding.
With this perspective, we regard the row and column dividers to be the horizontal and vertical lines at half-integer positions that partition the plot of a permutation into cells.
We follow this convention.
With a minor abuse of terminology, we often refer to the matrix entries themselves as cells, calling $\cBlank$ entries \emph{blank} cells, and $\cUp$ and $\cDown$ entries \emph{non-blank} cells.

The \emph{grid class} $\Grid(M)$ then consists of those permutations that have at least one $M$\!-gridding.
We use $\Grid_n(M)$ to denote the set of permutations of length $n$ in $\Grid(M)$.
A~permutation together with one of its $M$\!-griddings is called an \emph{$M$\!-gridded permutation} (or just a \emph{gridded permutation} if the gridding matrix is clear from the context).
The \emph{gridded class}, consisting of all $M$\!-gridded permutations, is denoted by $\Gridhash(M)$,
and the set of $M$\!-gridded permutations of length $n$ by $\Gridhash_n(M)$.
When we use $\sigma^\#$ to denote a gridded permutation, its underlying permutation is~$\sigma$.

Observe that if a gridding matrix $M$ has dimensions $r\times s$, then no permutation in $\Grid_n(M)$ can have more than $(n+1)^{r+s-2}$ griddings,
because there are only $n+1$ possible positions for each row and column divider.
Since this is polynomial in~$n$,
the exponential growth rate of the gridded class is equal to that of the grid class~\cite[Proposition 2.1]{Vatter2011}. 

Building on this observation,
our strategy for the asymptotic enumeration of $\Grid(M)$ consists of the following five steps:
\begin{enumerate}
  \item\label{step2}
  Find the proportion of points that occur in each cell in a typical large $M$\!-gridded permutation. 

  \item
  Determine the asymptotic enumeration of the corresponding
  gridded class:
  \[\big|\Gridhash_n(M)\big| \;\sim\; \theta^\#(n)\,g^n ,\]
  where $g$ is the exponential {growth rate} of the class, and $\theta^\#(n)$ is subexponential.

\item\label{step3}
  Determine, for each $\ell\geqs1$, how a typical large $M$\!-gridded permutation $\sigma^\#$ must be structured so that its underlying permutation $\sigma$ has exactly $\ell$ distinct $M$\!-griddings.

  \item
  Let $\bshn$ be a gridded permutation drawn uniformly at random from $\Gridhash_n(M)$, and let $\bsn$ be its underlying permutation.
  By combining steps \ref{step2} and \ref{step3}, calculate, for each $\ell\geqs1$, the asymptotic probability
  \[
  P_\ell \defeq \liminfty \prob{\text{$\bsn$ has exactly $\ell$ distinct $M$\!-griddings}} .
  \]

  \item
  Let
  \[
  \kappa_M
  \eq\sum_{\ell\geqs1}P_\ell/\ell
  \eq \liminfty \frac{\big|\Grid_n(M)|}{\big|\Gridhash_n(M)|}.
  \]
  Then $\big|\Grid_n(M)\big|\sim \kappa_M\,\theta^\#(n)\,g^n$.
\end{enumerate}

The rest of the paper is structured as follows:
\begin{itemize}
  \item In Section~\ref{sectSkinny}, we consider skinny classes.
    We determine how points can ``dance'' between the cells in a typical large gridded permutation.
     Hence, we deduce the asymptotic enumeration of any skinny class (Theorem~\ref{thmSkinnyAsympt}).
  \item Section~\ref{sectDistribution} concerns the distribution of points between the cells in a typical large gridded permutation in a connected class (Theorem~\ref{thmConnectedDistrib}).
      This result is then used to determine the shape of a typical permutation in such a class (Theorem~\ref{thmLimitShape}).
  \item In Section~\ref{sectLTX}, we apply our strategy to connected one-corner classes.
      We begin by finding the proportion of points in each cell in a typical large gridded permutation (Section~\ref{sectLTXProportions}).
      Then we determine the asymptotic enumeration of gridded connected one-corner classes (Proposition~\ref{propGriddedAsymptoticsLTX}).
      The heart of this section consists of a detailed analysis of how typical permutations in these classes can be gridded, based on an examination of the eleven possible corner types (Sections~\ref{sectLTXDancing}--\ref{sectLTXConstrained}).
      This then yields the asymptotic enumeration of any connected one-corner class (Theorem~\ref{thmLTXAsympt}).
      We conclude by briefly considering the application of our approach to classes with more than one corner (Section~\ref{sectBeyondLTX}).
\end{itemize}

\section{Skinny classes}\label{sectSkinny}

We begin with the simplest case, for which each of the five steps in the analysis is easy.
Recall that $\Grid(M)$ is skinny if $M$ is simply a $\cUp/\cDown$ vector.
Although, as mentioned above, a procedure is known which gives the generating function of any skinny class, this would be a very inefficient way of determining their asymptotic enumeration.

We begin by exactly enumerating skinny \emph{gridded} classes. 

\begin{prop}\label{propSkinnyGriddedCount}
    If $\Grid(M)$ is a $k$-cell skinny grid class, then $\big|\Gridhash_n(M)\big|=k^n$.
\end{prop}
\begin{proof}
  Any $M$\!-gridded permutation can be uniquely constructed from the empty $M$\!-gridded permutation by repeatedly adding a new maximum point.
  This point may be placed in any one of the $k$ cells.
  Moreover, there is only one way of adding a maximum to any particular cell, because of the monotonicity constraints.
  See Figure~\ref{figSkinnyGrowth} for an illustration.
\end{proof}

\begin{figure}[ht]
  \centering
  \begin{tikzpicture}[scale=0.225] \plotgriddedperm{7}{1,4,5,2,6,7,3}{3,5}{} \end{tikzpicture}
  ~ \raisebox{24pt}{$\;\mapsto\;$} ~
  \begin{tikzpicture}[scale=0.225] \plotgriddedperm{8}{1,4,5,8,2,6,7,3}{4,6}{} \circpt{4}{8} \end{tikzpicture} ~
  \begin{tikzpicture}[scale=0.225] \plotgriddedperm{8}{1,4,5,2,6,8,7,3}{3,6}{} \circpt{6}{8} \end{tikzpicture} ~
  \begin{tikzpicture}[scale=0.225] \plotgriddedperm{8}{1,4,5,2,6,8,7,3}{3,5}{} \circpt{6}{8} \end{tikzpicture}
  \caption{The three ways of adding a new maximum point to a gridded permutation in the skinny class $\Gridhash(\!\gcone{3}{1,1,-1}\!)$}\label{figSkinnyGrowth}
\end{figure}

\subsection{Peaks and peak points}

We now introduce some concepts that are important when analysing how permutations can be gridded.
Firstly, given a skinny grid class $\Grid(M)$, where $M$ is the $\cUp/\cDown$
vector $(m_1,\ldots,m_k)$,
we say that a pair $(m_i,m_{i+1})$ of adjacent cells forms a \emph{peak} if $m_{i+1}\neq m_i$.
That is, a peak either looks like $\!\gcone{2}{1,-1}\!$, which we say \emph{points up}, or else looks like or $\!\gcone{2}{-1,1}\!$, which we say \emph{points down}.
For example, $\Grid(\!\gcone{5}{1,-1,1,1,-1}\!)$ has three peaks, two pointing up and one pointing down.

Secondly, suppose we have a skinny grid class $\Grid(M)$ with a 
peak~$\Lambda$,
and that $\sigma^\#$ is an $M$\!-gridded permutation with at least two points in each cell, witnessing the orientation of the cell.
Then the \emph{peak point} of $\Lambda$ is the highest of the points of $\sigma^\#$ in the two cells of $\Lambda$ if $\Lambda$ points up,
and is the lowest of the points of $\sigma^\#$ in 
$\Lambda$ if $\Lambda$ points down.
For example, in the rightmost two gridded permutations in Figure~\ref{figSkinnyGrowth}, the peak point (of the only peak in the class) is circled.

Note that our use of the term ``peak'' is nonstandard, differing from its traditional use to denote a consecutive 132 or 231 pattern in a permutation.
For us a \emph{peak} is part of a gridding matrix.
On the other hand, a \emph{peak point} is the central point of either a peak or valley (in the traditional sense) in a \emph{gridded} permutation.
These concepts are generalised further below, when we consider non-skinny classes.

\subsection{Dancing and constrained gridded permutations}\label{sectDancing}

Suppose $\Grid(M)$ is skinny and $\sigma^\#\in\Gridhash(M)$.
If $Q$ is a peak point of $\sigma^\#$, then $Q$ is immediately adjacent to a column divider.
That is, there is no other point between $Q$ and the divider.
The movement of this divider to the other side of $Q$ results in another valid $M$\!-gridding of~$\sigma$.
We say that $Q$ can \emph{dance} and that this new gridded permutation is the result of $Q$ \emph{dancing}.
Conceptually, we consider $Q$ as dancing from one cell to an adjacent one (although in fact it is the divider rather than the point that moves).
See Figure~\ref{figSkinnyDance} for an illustration.
The notion of dancing is generalised below in Section~\ref{sectLTXDancing}, when we consider non-skinny classes.

\begin{figure}[ht]
  \centering
  \begin{tikzpicture}[scale=0.19] \plotgriddedperm{12}{2,8,9,11,10,5,3,4,6,12,1,7}{3,6,10}{} \circpt{4}{11}\circpt{7}{3} \end{tikzpicture} ~~~~~~~~
  \begin{tikzpicture}[scale=0.19] \plotgriddedperm{12}{2,8,9,11,10,5,3,4,6,12,1,7}{4,6,10}{} \circpt{4}{11}\circpt{7}{3} \end{tikzpicture} ~~~~~~~~
  \begin{tikzpicture}[scale=0.19] \plotgriddedperm{12}{2,8,9,11,10,5,3,4,6,12,1,7}{3,7,10}{} \circpt{4}{11}\circpt{7}{3} \end{tikzpicture} ~~~~~~~~
  \begin{tikzpicture}[scale=0.19] \plotgriddedperm{12}{2,8,9,11,10,5,3,4,6,12,1,7}{4,7,10}{} \circpt{4}{11}\circpt{7}{3} \end{tikzpicture}
  \caption{The four distinct griddings of a permutation in $\Grid(\!\gcone{4}{1,-1,1,1}\!)$; points which can dance are circled}
  \label{figSkinnyDance}
\end{figure}

The key to our approach is the fact that
for most permutations the gridding possibilities are heavily restricted.
To formalise this observation, we focus our attention on certain well-behaved gridded permutations
for which the only way to construct another gridding of the underlying permutation is through the dancing of peak points.
Suppose $\Grid(M)$ is skinny and $\sigma^\#\in\Gridhash(M)$. 
We say that $\sigma^\#$ is \emph{$M$\!-constrained} (or just \emph{constrained}) if
\begin{itemize}
  \item[(a)] every $M$\!-gridding of its underlying permutation $\sigma$ is the result of zero or more peak points of $\sigma^\#$ dancing,
  and
  \item[(b)] in every $M$\!-gridding of $\sigma$, each 
  cell contains at least two points.
\end{itemize}

For constrained gridded permutations, counting possible griddings is easy.

\begin{prop}\label{propSkinnyGriddings}
    If $\Grid(M)$ is a skinny grid class with $p$ peaks, and $\sigma^\#\in\Gridhash(M)$ is $M$\!-constrained, then $\sigma$ has exactly $2^p$ distinct $M$\!-griddings.
\end{prop}
\begin{proof}
Since $\sigma^\#$ is constrained, every $M$\!-gridding of $\sigma$ is the result of zero or more peak points of $\sigma^\#$ dancing.
Since each cell of $\sigma^\#$ contains at least two points, $\sigma^\#$ has $p$ distinct peak points (one for each peak), which can dance independently.
For each peak point of $\sigma^\#$, we can choose whether it dances or not,
yielding a total of $2^p$ distinct $M$\!-griddings for $\sigma$.
\end{proof}

The following proposition gives sufficient conditions for a gridded permutation in a skinny class to be constrained.

\begin{prop}\label{propSkinnyConstrained}
  Suppose $\Grid(M)$ is skinny and $\sigma^\#\in\Gridhash(M)$ is such that
  each 
  cell contains at least four points and
  there are no two adjacent cells whose contents together form an increasing or decreasing sequence.
  Then $\sigma^\#$ is $M$\!-constrained.
\end{prop}
\begin{proof}
The contents of each cell of $\sigma^\#$ consists of an increasing or decreasing sequence of 
points.
However, there is no pair of adjacent cells whose contents together form an increasing or decreasing sequence.

Thus, in any $M$\!-gridding of $\sigma$, by the monotonicity constraints,
there must be a divider adjacent to each peak point of $\sigma^\#$.
We also claim that there must be
a divider between each pair of adjacent cells of $\sigma^\#$ that have the same orientation.
Any other attempt at gridding
(such as placing the two points next to the divider in a cell with the opposite orientation) would require the use of additional dividers.
However, the number of dividers separating the cells must be one less than the number of cells.

Thus, any $M$\!-gridding of $\sigma$ can be formed from $\sigma^\#$ by zero or more of its peak points dancing.

Moreover, in any $M$\!-gridding of $\sigma$,
at most two points from any cell of $\sigma^\#$ (the first and the last) can be gridded in another cell.
So each cell of any $M$\!-gridding of $\sigma$ contains at least two points.
\end{proof}

Constrained gridded permutations are of interest not only because it is easy to count their griddings, but also because \emph{almost all} gridded permutations are constrained.

\begin{prop}\label{propSkinnyConstrainedAreGeneric}
If $\Grid(M)$ is skinny, then almost all $M$\!-gridded permutations are $M$\!-constrained.
That is, if $\bshn$ is drawn uniformly at random from $\Gridhash_n(M)$, then
\[
\liminfty \prob{\text{$\bshn$ is $M$\!-constrained}} \eq 1 .
\]
\end{prop}
\begin{proof}
  Suppose $M$ has $k$ cells.
  Recall that constructing an $n$-point $M$\!-gridded permutation is equivalent to choosing for each $i$ from $1$ to $n$, in which of the $k$ cells to place point $i$ (where points are numbered from bottom to top).
  Thus, the number of $M$\!-gridded $n$-permutations with exactly $m$ points in a given cell equals
  \[
  \binom{n}m(k-1)^{n-m} \;<\; n^m(k-1)^n .
  \]
  Here, $\binom{n}m$ is the number of ways of choosing the points in the given cell, and $(k-1)^{n-m}$ is the number of ways of distributing the remaining points.
  So the total number of $n$-point $M$\!-gridded permutations with fewer than four points in some cell
  is less than $4kn^3(k-1)^n$,
  there being four choices for the value of~$m$,
  and $k$ choices of cell.

  Similarly, the number of $n$-point $M$\!-gridded permutations with a given pair of adjacent cells forming an increasing or decreasing sequence of length $m$ is less than
  \[
  (m+1)(k-1)^n \;<\; (n+1)(k-1)^n .
  \]
  Here,
  $(k-1)^n$ is an upper bound on the number of ways of distributing the points if we consider the pair of adjacent cells merged to form a single ``super cell'', 
  and the factor $m+1$ is the number of choices for the position of the divider that splits the pair of cells.
  So the total number of $n$-point $M$\!-gridded permutations with two adjacent cells forming an increasing or decreasing sequence
  is less than $(n+1)^2(k-1)^{n+1}$,
  there being $n+1$ choices for the value of~$m$,
  and $k-1$ choices for the pair of cells.

  Thus, by Propositions~\ref{propSkinnyGriddedCount} and~\ref{propSkinnyConstrained}, the proportion of $n$-point $M$\!-gridded permutations which are not constrained is less than
  \[
  \frac{4kn^3(k-1)^n \:+\: (n+1)^2(k-1)^{n+1}}{k^n}
  \;<\;
  5k(n+1)^3\Big(1-\frac1k\Big)^{\!n},
  \]
  which converges to zero as $n$ tends to infinity.
\end{proof}

The asymptotic enumeration of skinny grid classes then follows directly from Propositions~\ref{propSkinnyGriddedCount}, \ref{propSkinnyGriddings} and~\ref{propSkinnyConstrainedAreGeneric}.

\begin{thm}\label{thmSkinnyAsympt}
    If $\Grid(M)$ is a $k$-cell skinny grid class with $p$ peaks, then $\big|\Grid_n(M)\big| \sim 2^{-p}k^n$.
\end{thm}
\begin{proof}
For almost all of the $k^n$ distinct $M$\!-gridded permutations, the underlying permutation has exactly $2^p$ distinct $M$\!-griddings.
\end{proof}

\section{The distribution of points between cells}\label{sectDistribution}

As we have seen, almost every large permutation in a given skinny class has the same number of griddings (that depends only on the number of peaks).
In non-skinny classes this is not the case:
the number of griddings may depend on the structure of the permutation.
In order to determine the asymptotic probability that a permutation has a specific number of griddings,
we need to know the proportion of points that occur in each cell in a typical large gridded permutation.
This is the focus of this section.

Much of the analysis required can be found in the paper of Albert and Vatter~\cite{AV2019} and (in a more sketchy form) in the second author's PhD thesis~\cite[Chapter~6]{BevanThesis}.
However, neither of these works contains all that we need, so we reproduce the argument in full here.
Our presentation combines ideas from both approaches.

We begin by introducing a family of matrices to record the number of points or proportion of points in each cell (Section~\ref{sectMAdmissibleMatrices}).
We then determine the asymptotic number of gridded permutations with a given distribution (Section~\ref{sectCountnGammaGriddings}).
This is followed by calculating the distribution of points for which this enumeration is the greatest (Section~\ref{sectMaximalMDistribution}).
We then prove 
that this distribution is unique if the class is connected (Section~\ref{sectMaximalMDistributionUnique}).
Finally, we establish in Theorem~\ref{thmConnectedDistrib} that the distribution of points in almost all gridded permutations
is close to this maximal distribution, 
This is then used to determine the limit shapes of connected classes (Theorem~\ref{thmLimitShape}). 

\subsection{\texorpdfstring{$M$\!-admissible and $M$\!-distribution matrices}{M-admissible and M-distribution matrices}}\label{sectMAdmissibleMatrices}

Given a gridding matrix $M$, we say that a nonnegative real matrix $A$ of the same dimensions as $M$ is $M$\!-\emph{admissible} if $A_{i,j}$ is zero whenever $M_{i,j}$ is blank.
We refer to the sum of the entries of such a matrix as its \emph{weight}, and use $\|A\|$ to denote the weight of~$A$.
Note that, unlike in the definition of a gridding on page~\pageref{defMGridding}, we index matrices in the traditional manner in the following discussion.

We use \emph{integer} $M$\!-admissible matrices to record the number of points in each cell.
Suppose $A=(a_{i,j})$ is an integer $M$\!-admissible matrix of weight $n$.
Then $\Gridhash_A(M)$ denotes the \mbox{subset} of $\Gridhash_n(M)$ consisting of those $M$\!-gridded permutations with $a_{i,j}$ points in cell $(i,j)$, for each~$(i,j)$.
For example, the leftmost gridded permutation in Figure~\ref{figGriddings} on page~\pageref{figGriddings} is an element of 
\[\Gridhash_{\begin{smallmx}2&3&3\\0&1&0\end{smallmx}}\!\big(\!\gctwo{3}{-1,-1,1}{0,1}\!\big)\! .\]

Thus $\Gridhash_n(M)$ can be partitioned into subsets as follows:
\[
\Gridhash_n(M) \eq \biguplus_{\|A\|=n} \Gridhash_A(M),
\]
where the disjoint union is over all integer $M$\!-admissible matrices of weight $n$.
The number of gridded permutations in one of these subsets is given by the following product of multinomial coefficients.

\begin{prop}\label{propCountAGriddings}
  Suppose $A=(a_{i,j})$ is an integer $M$\!-admissible matrix with dimensions $r\times s$. Then,
  \[
  \big|\Gridhash_A(M)\big| \eq
  \prod_{i=1}^r \binom{\sum_{j=1}^s a_{i,j}}{a_{i,1},a_{i,2},\ldots,a_{i,s}}
  \:\times\:
  \prod_{j=1}^s \binom{\sum_{i=1}^r a_{i,j}}{a_{1,j},a_{2,j},\ldots,a_{r,j}} .
  \]
\end{prop}

\begin{proof}
The ordering of points (increasing or decreasing) \emph{within} a particular cell is fixed by the corresponding entry of $M$.
However, the interleaving of points in \emph{distinct} cells in the same row or column can be chosen arbitrarily and independently.
The multinomial coefficient in the first product counts the number of ways of vertically interleaving the points in the cells in row~$i$.
Similarly, the term in the second product counts the number of ways of horizontally interleaving the points in the cells in column~$j$.
\end{proof}

We use $M$\!-admissible matrices \emph{of weight one} to record the \emph{proportion} of points in each cell.
We call such matrices \emph{$M$\!-distribution matrices}.
To avoid having to worry about rounding the number of points in each cell to an integer, we make use of Baranyai's Rounding Lemma~\cite{Baranyai1975}, as a consequence of which we have the following result.

\begin{prop}\label{propBaranyai}
  Suppose $\Gamma=(\gamma_{i,j})$ is an $M$\!-distribution matrix, and $n$ is any positive integer.
  Then there exists an integer $M$\!-admissible matrix $A=(a_{i,j})$ of weight $n$ such that, for each $i,j$, we have
  $|a_{i,j}-n\gamma_{i,j}|<1$.
\end{prop}

In light of this, if $\Gamma=(\gamma_{i,j})$ is an $M$\!-distribution matrix and $n$ a positive integer,
we use $\Gridhash_{\Gamma n}(M)$ to denote $\Gridhash_A(M)$, where $A$ is some integer $M$\!-admissible matrix of weight $n$ each of whose entries differs from the corresponding entry of $n\Gamma$ by less than one, the existence of such an $A$ being guaranteed by Proposition~\ref{propBaranyai}.
The exact choice of $A$ is of no consequence to our arguments.
Note that in any gridded permutation in $\Gridhash_{\Gamma n}(M)$ the proportion of points in cell $(i,j)$ differs from $\gamma_{i,j}$ by less than $1/n$.
Let $\Gridhash_\Gamma(M)=\bigcup_{n\geqs0}\Gridhash_{\Gamma n}(M)$.

\subsection{Asymptotics of gridded permutations with a given distribution}
\label{sectCountnGammaGriddings}

Suppose $\tau=\gamma_1+\gamma_2+\ldots+\gamma_k$, where each $\gamma_i>0$.
Then Stirling's approximation gives the following asymptotic form for a multinomial coefficient.
\begin{equation}\label{eqMultinomial}
\binom{\tau\+n}{\gamma_1 n,\gamma_2 n,\ldots,\gamma_k n}
\;\sim\;
\sqrt{\frac{\tau}{(2\+\pi)^{k-1}\+\gamma_1\+\gamma_2\ldots\gamma_k}}
\+
n^{-(k-1)/2}
\+
\left(\frac{\tau^\tau}{\gamma_1^{\gamma_1}\+\gamma_2^{\gamma_2}\ldots\+\gamma_k^{\gamma_k}}\right)^{\!n}
.
\end{equation}
In addition, the factorial bounds $\sqrt{2\pi n}(n/e)^n < n! \leqslant e\sqrt{n}(n/e)^n$, valid for all positive $n$,
imply that the multinomial coefficient is always bounded above by a constant multiple of the right hand side of~\eqref{eqMultinomial}.
Observe also that if each $\gamma_in\geqs1$ and $\tau\leqs1$ then $\frac{\tau}{\gamma_1\+\gamma_2\ldots\gamma_k}\leqs n^k$.

Hence, by Proposition~\ref{propCountAGriddings}, we have the following asymptotic enumeration of $\Gridhash_\Gamma(M)$ and bound on the size of $\Gridhash_{\Gamma n}(M)$.
\begin{prop}\label{propCountnGammaGriddings}
If $\Gamma=(\gamma_{i,j})$ is an $M$\!-distribution matrix with
row sums $\rho_i=\sum_j\gamma_{i,j}$
and
column sums $\kappa_j=\sum_i\gamma_{i,j}$,
then
\[
\big|\Gridhash_{\Gamma n}(M)\big| \;\sim\; C \+ n^\beta \+ g^n ,
\]
where
\[
g \eq g(\Gamma) \defeq
\prod_i \frac{\rho_i^{\rho_i}}{\prod_j \gamma_{i,j}^{\gamma_{i,j}}}
\:\times\:
\prod_j \frac{\kappa_j^{\kappa_j}}{\prod_i \gamma_{i,j}^{\gamma_{i,j}}}
,
\]
and $C$ and $\beta$ are constants that only depend on~$\Gamma$.

Moreover, there exist constants $C'$ and $\beta'$ that only depend on~$M$, such that for all $n$,
\[
\big|\Gridhash_{\Gamma n}(M)\big| \;<\; C' \+ n^{\beta'} \+ g^n .
\]
\end{prop}
Note that $g(\Gamma)=\liminfty\big|\Gridhash_{\Gamma n}(M)\big|^{1/n}=\gr\big(\Gridhash_\Gamma(M)\big)$ is the exponential growth rate of those $M$\!-gridded permutations whose points are distributed between the cells in the proportions specified by $\Gamma$.

To avoid problems below, we take the products in the denominators of the expression for $g(\Gamma)$ to be over nonzero entries of $\Gamma$ only. 

\subsection{Maximising the growth rate}\label{sectMaximalMDistribution}

Given a grid class $\Grid(M)$,
we would like to find an $M$\!-distribution matrix $\Gamma=(\gamma_{i,j})$ for which the growth rate $g(\Gamma)$ is greatest.
We call such a matrix a \emph{maximal} $M$\!-distribution matrix.
The following result specifies equations that must be satisfied by the entries of such a matrix.

\begin{prop}\label{propMaxGR}
  Suppose $\Gamma=(\gamma_{i,j})$ is a maximal $M$\!-distribution matrix.
  Then there exists a constant $\lambda$ such that, for each nonzero entry $\gamma_{i,j}$ of $\Gamma$, we have
  \[
  \frac{\gamma_{i,j}^2}{\rho_i\+\kappa_j} \eq \lambda ,
  \]
  where $\rho_i=\sum\limits_j\gamma_{i,j}$ and $\kappa_j=\sum\limits_i\gamma_{i,j}$ are the row and column sums of $\Gamma$.
\end{prop}

\begin{proof}
We want to determine which choices of values for the $\gamma_{i,j}$ maximise $g(\Gamma)$, subject to the requirement that $\sum\gamma_{i,j}=1$.
This is a constrained optimisation problem, so we use the method of Lagrange multipliers.
To simplify the algebra, we maximise $\log g(\Gamma)$, which is of course equivalent to maximising $g(\Gamma)$.
Viewing the nonzero $\gamma_{i,j}$ as formal variables, we introduce the auxiliary function
\[
L(\Gamma,\mu) \eq \log g(\Gamma) \;-\; \mu\big(1-\sum\gamma_{i,j}\big) .
\]
Then $\Gamma$ must satisfy
\[
\frac{\partial}{\partial \mu} L(\Gamma,\mu) \eq 0
\qquad \text{and} \qquad
\frac{\partial}{\partial \gamma_{i,j}} L(\Gamma,\mu) \eq 0,
\]
for each nonzero $\gamma_{i,j}$.
The first of these requirements is simply 
the constraint that $\sum\gamma_{i,j}=1$ (that is, $\Gamma$ has weight one).

Now,
\[
\log g(\Gamma) \eq
\sum_i \Big({{\rho_i}\log\rho_i}\,-\,{\sum_j {\gamma_{i,j}}\log\gamma_{i,j}}\Big)
\:+\:
\sum_j \Big({{\kappa_j}\log\kappa_j}\,-\,{\sum_i {\gamma_{i,j}}\log\gamma_{i,j}}\Big)
.
\]
Given $i$ and $j$, for brevity let $\gamma=\gamma_{i,j}$.
Also let $\overline{\rho}=\rho_i-\gamma$ and $\overline{\kappa}=\kappa_j-\gamma$ be the sum of the entries other than $\gamma$ in row $i$ and in column $j$ of $\Gamma$, respectively.
Both $\overline{\rho}$ and $\overline{\kappa}$ are independent of~$\gamma$, as are the contributions to $L(\Gamma,\mu)$ from any other row or column.
Then,
\begin{align*}
\frac{\partial}{\partial \gamma_{i,j}} L(\Gamma,\mu) 
\eq &
\frac{\partial}{\partial \gamma} \Big( {(\overline{\rho}+\gamma)}\log(\overline{\rho}+\gamma) \:+\: {(\overline{\kappa}+\gamma)}\log(\overline{\kappa}+\gamma) \:-\: 2{\gamma}\log\gamma \:+\: \mu\gamma \Big) \\[3pt]
\eq & \big( 1 + \log(\overline{\rho}+\gamma) \big) \:+\: \big( 1 + \log(\overline{\kappa}+\gamma) \big) \:-\: 2\big( 1 + \log\gamma \big) \:+\: \mu \\[3pt]
\eq & \log\rho_i \:+\: \log\kappa_j \:-\: 2\log\gamma_{i,j} \:+\: \mu .
\end{align*}
Thus,
\[
\frac{\partial}{\partial \gamma_{i,j}} L(\Gamma,\mu) \eq 0
\quad \text{only if} \quad
\frac{\gamma_{i,j}^2}{\rho_i\+\kappa_j} \eq e^\mu .
\]
The result follows by letting $\lambda=e^\mu$.
\end{proof}

\subsection{Uniqueness of the maximal distribution for connected classes}\label{sectMaximalMDistributionUnique}

Albert and Vatter~\cite{AV2019} give a different characterisation of maximal $M$\!-distribution matrices (in a more general context),
using the theory of singular value decompositions of matrices.

Given a real matrix $A$, a non-negative real number $s$ is a \emph{singular value} for $A$ if there exist unit vectors $\mathbf{u}$ and~$\mathbf{v}$ such that
$A\mathbf{v}=s\mathbf{u}$ and $A^{\mathsf{T}}\mathbf{u}=s\mathbf{v}$.
The vectors $\mathbf{u}$ and $\mathbf{v}$ are called \emph{left singular} and \emph{right singular} vectors for $s$, respectively.
If $s$ is a {singular value} for $A$ then $s^2$ is an eigenvalue of both $AA^{\mathsf{T}}$ and $A^{\mathsf{T}}\!A$.
Moreover, every left singular vector for $s$ is an eigenvector of $AA^{\mathsf{T}}$ for the eigenvalue $s^2$,
and every right singular vector for $s$ is an eigenvector of $A^{\mathsf{T}}\!A$ for the eigenvalue $s^2$.

\newcommand{\Moi}{M^{\scriptscriptstyle{\mathsf{01}}}}

Given a gridding matrix $M$, let $\Moi$ be the binary 
matrix such that $\Moi_{i,j}=0$ if $M_{i,j}$ is blank, and $\Moi_{i,j}=1$ otherwise.
Maximal $M$\!-distribution matrices can be characterised as follows.

\begin{prop}[{Albert and Vatter~\cite{AV2019}}]\label{propSVD}
  Suppose $\Gamma=(\gamma_{i,j})$ is a maximal $M$\!-distribution matrix.
  Then $\gamma_{i,j} = \Moi_{i,j}\,\mathbf{v}_i\mathbf{u}_j/s$, where $s$ is the greatest singular value of $\Moi$, and $\mathbf{u}$ and $\mathbf{v}$ are left and right singular vectors for $s$, respectively.
\end{prop}

In general, left and right singular vectors may not be unique, so $\Grid(M)$ may have multiple maximal $M$\!-distribution matrices.
However, if we limit which classes we consider, uniqueness can be guaranteed.

Given a gridding matrix $M$, its \emph{cell graph} is the graph whose vertices are its non-blank cells,
and in which two vertices are adjacent if they share a row or a column and all the cells between them are blank.

For example, the cell graph of \gcthree{4}{1,-1,0,1}{0,0,-1}{0,-1,1} is
\raisebox{-6pt}{
\begin{tikzpicture}[scale=0.25]
  \plotpermnobox{}{3,3,0,3}
  \plotpermnobox{}{0,0,2}
  \plotpermnobox{}{0,1,1}
  \draw (1,3)--(4,3);
  \draw (2,3)--(2,1)--(3,1)--(3,2);
\end{tikzpicture}}.

If the cell graph of a gridding matrix is connected, then we also say that the matrix and the corresponding grid class are \emph{connected}.

A connected grid class has a unique maximal distribution matrix.
To prove that this is the case, we extend the notion of connectivity to real matrices.
We say that a nonnegative matrix $A$ is \emph{connected} if there do not exist permutation matrices $P$ and $Q$ such that $PAQ=\begin{smallmx}R&0\\0&S\end{smallmx}$ for some non-trivial $R$ and~$S$.

\begin{prop}\label{propUniqueSVD}
  If a nonnegative matrix $A$ is connected, then its greatest singular value $s$ 
  has unique left and right singular vectors $\mathbf{u}$ and $\mathbf{v}$, and
  the entries of $\mathbf{u}$ and $\mathbf{v}$ are all positive.
\end{prop}
\begin{proof}
Consider $A$ to be a weighted biadjacency matrix of a bipartite graph $G$.
If $A$ is connected then so is $G$.
Both $A A^{\mathsf{T}}$ and $A^{\mathsf{T}}\!A$ are then symmetric weighted adjacency matrices of the projections $G_1$ and $G_2$ of $G$ onto its two vertex sets (in which nodes are adjacent if they share a neighbour in $G$). If $G$ is connected, then so are $G_1$ and $G_2$. Thus $A A^{\mathsf{T}}$ and $A^{\mathsf{T}}\!A$ are both irreducible.


  The result then follows from the Perron--Frobenius Theorem, given that $s^2$ is the principal eigenvalue of both $AA^{\mathsf{T}}$ and $A^{\mathsf{T}}\!A$,
  with $\mathbf{u}$ and $\mathbf{v}$ the corresponding (unique positive) principal eigenvectors.
\end{proof}

With this, we can prove that connectivity is a sufficient condition for a grid class to have a unique maximal $M$\!-distribution matrix.

\begin{prop}\label{propUniqueMaxDistrib}
  If $\Grid(M)$ is a connected grid class, then it has a unique maximal $M$\!-distribution matrix $\Gamma=(\gamma_{i,j})$.
  Moreover, $\gamma_{i,j}$ is positive if and only if $M_{i,j}$ is not blank.
\end{prop}
\begin{proof}
  Since $M$ is connected, then so is $\Moi$.
  Propositions~\ref{propSVD} and~\ref{propUniqueSVD} then guarantee that $\Gamma$ is unique, with $\gamma_{i,j} = \Moi_{i,j}\,\mathbf{v}_i\mathbf{u}_j/s$, where $s$ is the greatest singular value of $\Moi$, and $\mathbf{u}$ and $\mathbf{v}$ are the unique left and right singular vectors for $s$.
  Since $\mathbf{u}$ and $\mathbf{v}$ are positive, $\gamma_{i,j}>0$
  if and only if $M_{i,j}$ is not blank.
\end{proof}

If $M$ is connected, then we use $\Gamma_{\!M}$ to denote the unique maximal $M$\!-distribution matrix.

\subsection{The distribution of points in a typical gridded permutation}\label{sectTypicalDistribution}

We now prove that, if $\Grid(M)$ is connected, then the distribution of points between cells in
almost all $M$\!-gridded permutations is close to that specified by~$\Gamma_{\!M}$. 

Given any $\sigma^\#\in\Gridhash_n(M)$,
let $\sigma^\#_{(i,j)}$ denote the number of points of $\sigma^\#$ in cell $(i,j)$, and
let $\Gamma_{\!\sigma^\#}=\big(\sigma^\#_{(i,j)}/n\big)$ be the $M$\!-distribution matrix recording the proportion of the points of $\sigma^\#$ in each cell.

\begin{thm}\label{thmConnectedDistrib}
  If $\Grid(M)$ is connected and
  $\Gamma_{\!M}=(\gamma_{i,j})$
  is the unique maximal $M$\!-distribution matrix,
  then for any $\veps>0$,
  \[
  \liminfty \prob{ \max_{i,j} \big| \bsh_{(i,j)}/n - \gamma_{i,j} \big| \leqs \veps } \eq 1 ,
  \]
  where, for each $n$, the gridded permutation $\bsh$ is drawn uniformly at random from $\Gridhash_n(M)$.
\end{thm}

\begin{proof}
  For brevity, let $\Gamma=\Gamma_{\!M}$,
  and let $\SSS_n = \big\{
  \sigma^\#\in\Gridhash_n(M)
  \::\:
  \max_{i,j} \big| \sigma^\#_{(i,j)}/n - \gamma_{i,j} \big| \leqs \veps
  \big\}$
  consist of those $M$\!-gridded $n$-permutations whose distribution of points is \emph{$\veps$-close to $\Gamma$},
  the proportion of points in each cell differing by no more than $\veps$ from that specified by~$\Gamma$.

  Also, let $\overline{\SSS}_n = \Gridhash_n(M) \setminus \SSS_n$ be its complement,
  consisting of those $M$\!-gridded $n$-permutations whose distribution is \emph{$\veps$-far from~$\Gamma$},
  the proportion of points in some cell differing by more than~$\veps$ from that specified by~$\Gamma$.

  For sufficiently large $n$, we have $\Gridhash_{\Gamma n}(M) \subseteq \SSS_n$.
  Thus the exponential growth rate of $\SSS_n$ is bounded below by~$g(\Gamma)$:
  \[
  \gr(\SSS_n) \;\geqs\; \gr\big(\Gridhash_\Gamma(M)\big) \eq g(\Gamma) .
  \]

  Now let $\overline{\GGG}_n=\{\Gamma_{\!\sigma^\#}:\sigma^\#\in\overline{\SSS}_n\}$
  be the collection of $M$\!-distribution matrices that record the distribution of points in $M$\!-gridded $n$-permutations whose distribution is $\veps$-far from $\Gamma$.
  Then
  \[
  \overline{\SSS}_n \eq \bigcup_{\Gamma'\in\overline{\GGG}_n} \Gridhash_{\Gamma'n}(M) .
  \]
  Since $M$ is connected, $\Gamma$ is the only $M$\!-distribution matrix with maximal growth rate.
  Together with the continuity of the function $g(\cdot)$ and the compactness of the space of $M$\!-distribution matrices,
  this implies that
  there exists some $g_\veps<g(\Gamma)$ such that, for every $\Gamma'\in\bigcup_{n\geqs0}\overline{\GGG}_n$, we have $g(\Gamma')=\gr\big(\Gridhash_{\Gamma'}(M)\big) \leqs g_\veps$.

  Moreover, by the bound in Proposition~\ref{propCountAGriddings}, there exist constants $C'$ and $\beta'$, only dependent on~$M$,
  such that for each $\Gamma'$ and every $n$ we have $\big|\Gridhash_{\Gamma' n}(M)\big| < C' \+ n^{\beta'} \+ g_\veps^n$.
  Now $|\overline{\GGG}_n| < (n+1)^k$, where $k$ is the number of non-blank entries in~$M$.
  Hence the exponential growth rate of $\overline{\SSS}_n$ is strictly bounded above by~$g(\Gamma)$:
  \[
  \gr(\overline{\SSS}_n) \;\leqs\; g_\veps \;<\; g(\Gamma) .
  \]
  Thus
  $\liminfty |\overline{\SSS}_n|/|\SSS_n| = 0$ as required.
\end{proof}

\subsection{The shape of a typical permutation in a class}\label{sectLimitShapes}

Theorem~\ref{thmConnectedDistrib} and Proposition~\ref{propMaxGR}
enable us to calculate the asymptotic distribution of points in almost all gridded permutations in any 
connected class.
From this we can also determine what the plot of a typical large 
permutation in such a class looks like.
See Figure~\ref{figLimitShape} for an illustration.

\begin{figure}[ht]
  \centering
\includegraphics[width=.9in]{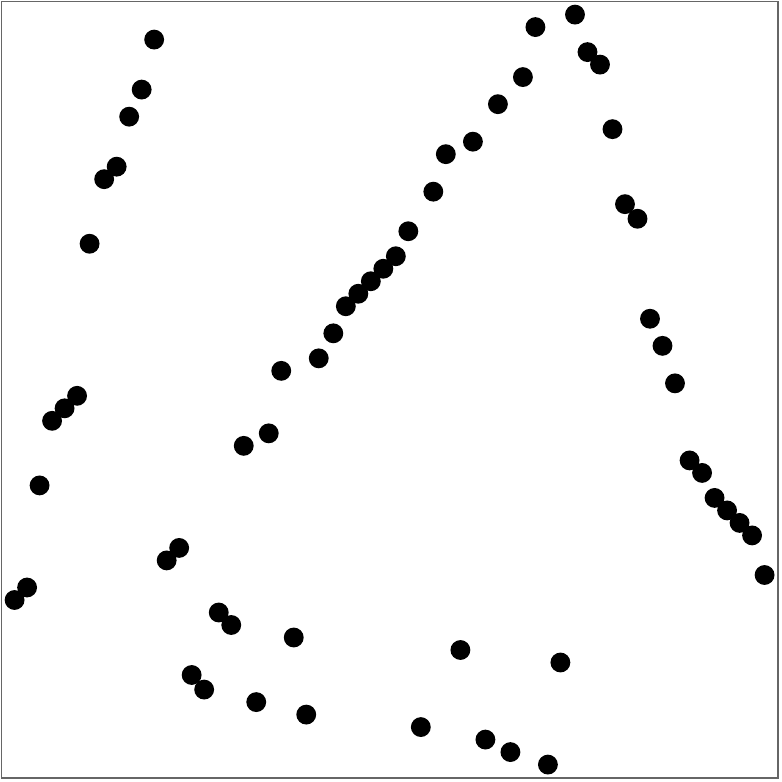}  \quad\;
\includegraphics[width=.9in]{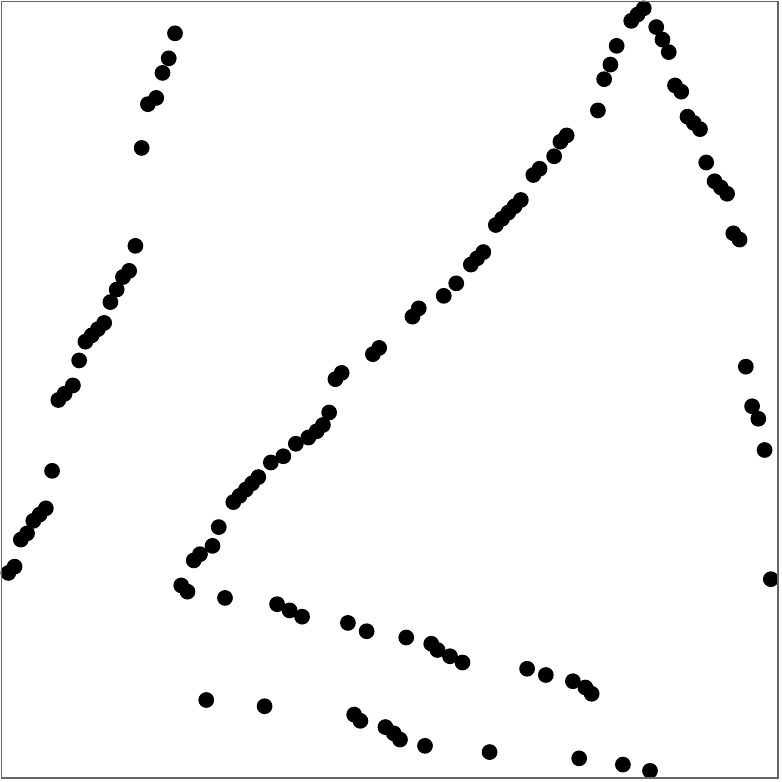} \quad\;
\includegraphics[width=.9in]{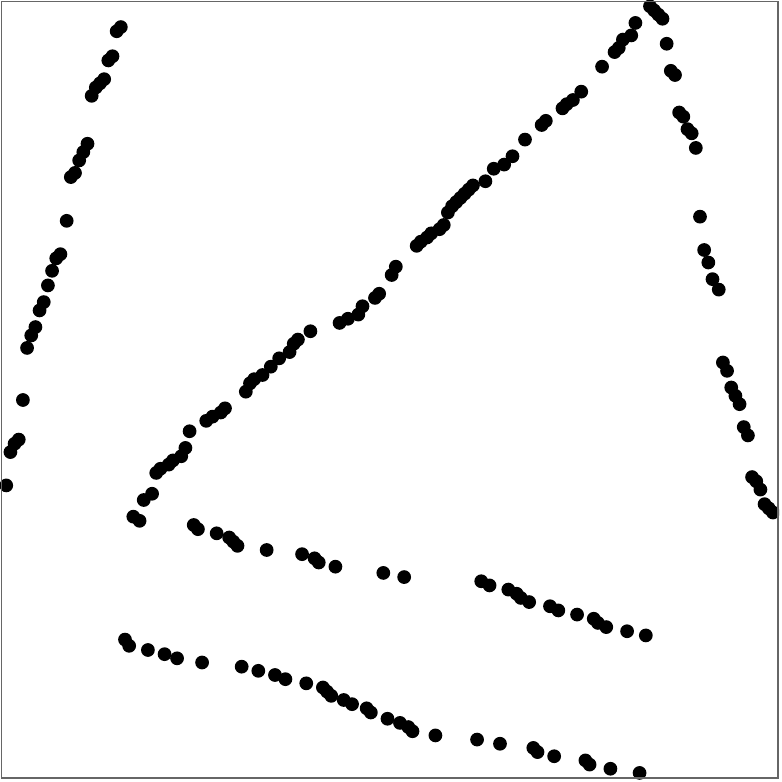} \quad\;
\includegraphics[width=.9in]{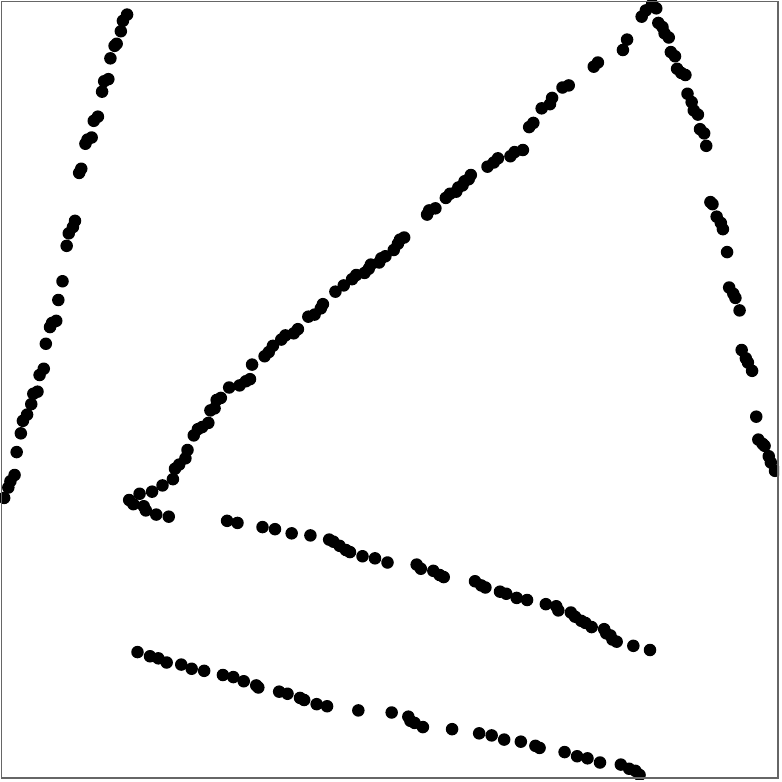}
\qquad\qquad
\includegraphics[width=.9in]{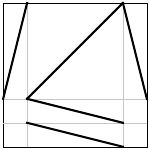}
  \caption{Plots of permutations of length 60, 120, 180 and 240 in $\Grid\big(\!\gcthree{3}{1,1,-1}{0,-1,0}{0,-1,0}\!\big)$, and the limit shape of the class}\label{figLimitShape}
\end{figure}

To formalise this idea,
we make use of certain probability measures on the unit square which act as analytic limits of sequences of permutations.
We briefly introduce only those parts of the theory that we need to state our result.
See the references below for more information.

A \emph{permuton} is a probability measure $\mu$ on the unit square $[0,1]^2$ with uniform marginals:
$  \mu ([a, b] \times [0,1]) \eq \mu ([0,1] \times [a, b]) \eq b - a $,
for every
$  0 \leqs a \leqs b \leqs 1 $.
There is a natural topology on the space of permutons obtained by restricting the weak topology on probability measures
(see~\cite{KKRW2020} for more details).

A permuton can be associated with a given permutation by taking its plot, scaling it into the unit square, and replacing the points with small squares,
as follows.
Suppose $A\subset[0,1]^2$ is a
rectangular region or a line segment.
Then we use
$\lambda_A$ to denote the unique probability measure on $[0,1]^2$ with support $A$ and mass uniformly distributed on~$A$.
If $\sigma$ is an $n$-permutation, then, for $i=1,\ldots,n$, let $S_\sigma(i)$ be the small square region $[(i-1)/n,i/n]\times[(\sigma(i)-1)/n,\sigma(i)/n]$.
Then
$\mu_\sigma = \sum_{i=1}^n \frac1n \lambda_{S_\sigma(i)}$
is the permuton corresponding to~$\sigma$.
See Figure~\ref{figPermPermuton} for an example.

\begin{figure}[ht]
  \centering
  \begin{tikzpicture}[line join=round,scale=0.25]
        \plotpermgrid{9}{3,1,4,5,9,2,6,8,7}
  \end{tikzpicture}
  \qquad \qquad \qquad
  \begin{tikzpicture}[line join=round,scale=0.25]
    \plotpermuton{9}{3,1,4,5,9,2,6,8,7}
  \end{tikzpicture}
  \caption{The plot of the permutation $\pi=314592687$, and a picture of its permuton $\mu_\pi$}\label{figPermPermuton}
\end{figure}

Permutons were introduced in \cite{HKMRS2013,HKMS2011} with a different but equivalent definition.
The measure-theoretic perspective was first used (a little earlier) in~\cite{PS2010},
and was later applied in \cite{GGKK2015}, where the term ``permuton'' first occurs.

Given a class of permutations $\CCC$, it may be that almost all large permutations in $\CCC$ have the ``same shape''.
Let $\bsn$ be an $n$-permutation drawn uniformly at random from~$\CCC_n$.
If the sequence of random permutons $(\mu_{\bsn})_{n\geqs1}$ converges in distribution for the weak topology to some (possibly random) permuton~$\mu$,
then we say that $\mu$ is the \emph{limit shape} of~$\CCC$.
See~\cite{BBFGMP2020} for a more detailed presentation.

If $M$ 
is connected,
then $\Gamma_M$ specifies the asymptotic distribution of points between the cells in a random permutation in $\Gridhash(M)$.
This also specifies the asymptotic position of the cell dividers.
The points in each cell are monotonic,
and points in cells that share a row or column may be arbitrarily interleaved,
so with high probability
no point is far from the line segment across one of the diagonals of its cell:

\begin{prop}\label{propLimit}
Let $\bsh$ be drawn uniformly at random from $\Gridhash_n(M)$, and $C$ be a cell of~$\bsh$.
Suppose $C$ contains $\gamma n$ points, and there are a total of $\rho n$ points in the row (or column) of cells including~$C$.
For each $k\in\{0,\ldots,\rho n\}$, let $S_k$ count how many of the lowest (or leftmost) $k$ points in the row (or column) lie in cell~$C$.
Then, for each $k$ and any $\veps>0$, we have
\[
\prob{\max\big\{ \big| S_k-\gamma k/\rho \big| \::\: k=0,\ldots,\rho n \big\} \;\geqs\; \veps n} \;\leqs\; 2\rho n e^{-2\veps(\rho-\gamma)n/\rho} .
\]
\end{prop}
\begin{proof}
  We apply Hoeffding's inequality when sampling without replacement~\cite[Section~6]{Hoeffding1963}.
  If $X_i$ is the indicator random variable for the $i$th point lying in~$C$, so $\prob{X_i=1}=\gamma/\rho$ and $S_k=\sum_{i=1}^k X_i$, then Hoeffding's inequality gives
  \[
  \prob{S_k\geqs\gamma k/\rho+\veps n} \;\leqs\; e^{-2(\veps n/k)^2k} \eq e^{-2\veps^2 n^2/k}.
  \]
  However, $S_k\leqs k$, so $S_k\geqs\gamma k/\rho+\veps n$ is only possible when $k\geqs\veps\rho n/(\rho-\gamma)$.
  Substitution of this lower bound for $k$ then yields
  \[
  \prob{S_k\geqs\gamma k/\rho+\veps n} \;\leqs\; e^{-2\veps(\rho-\gamma)n/\rho},
  \]
  valid for all~$k$.

  An analogous argument applies when considering the points from the top, giving the same bound on $\prob{S_k\leqs\gamma k/\rho-\veps n}$.
  The union bound then yields the claimed result.
\end{proof}

So, every point in a typical uniformly chosen $M$\!-gridded permutation is close to the cell diagonals.
Thus $\Gridhash(M)$ has a \emph{deterministic} limit shape $\mu_M$ formed of these diagonal line segments.
Moreover, $\Grid(M)$ has the same limit shape, since the number of griddings of an $n$-permutation is bounded by a polynomial in~$n$,
whereas only an exponentially small fraction of gridded permutations have a point at distance $\veps$ or more from the cell diagonals.

The limit shape $\mu_M$ is defined as follows.
As in the formal definition of an $M$\!-gridding on page~\pageref{defMGridding},
for simplicity, here
we index matrices in the Euclidean manner from the lower left corner, counting rows from the bottom,
and with the column index before the row index.
Suppose $\Gamma_M=(\gamma_{i,j})$ is the unique maximal $M$\!-distribution matrix.

Let $c_i=\sum_{i'\leqs i} \sum_j\gamma_{i',j}$ be the sum of the leftmost $i$ columns of~$\Gamma_M$,
and let $r_j=\sum_{j'\leqs j} \sum_i\gamma_{i,j'}$ be the sum of the lowermost $j$ rows of~$\Gamma_M$.
Then
$c_i$ gives the position of the $i$th column divider in~$\mu_M$, and
$r_j$ gives the position of the $j$th row divider.
Let $p_{i,j} = (c_i,r_j)$ be the point of their intersection.
Then the line segment in the support of $\mu_M$ for the non-blank cell $(i,j)$ is given by
\[
L(i,j) \eq
\begin{cases}
  \{(1-x)p_{i-1,j-1} \,+\, x p_{i,j} \::\: x\in[0,1]\}, & \text{if~} M_{i,j}=\cUp , \\[3pt]
  \{(1-x)p_{i-1,j} \,+\, x p_{i,j-1} \::\: x\in[0,1]\}, & \text{if~} M_{i,j}=\cDown .
\end{cases}
\]
Finally, the limit shape is constructed by distributing mass uniformly along the line segment for each non-blank cell:
\[
\mu_M \eq \sum_{\gamma_{i,j}>0} \gamma_{i,j} \+ \lambda_{L(i,j)} .
\]

Thus we have the following.

\begin{thm}\label{thmLimitShape}
  If $M$ 
  is connected,
  then $\mu_M$, defined above, is the deterministic limit shape of~$\Grid(M)$.
\end{thm}

For example, by applying Theorem~\ref{thmConnectedDistrib} and Proposition~\ref{propMaxGR}
we know that
the unique maximal $\!\gcthree{3}{1,1,-1}{0,-1,0}{0,-1,0}\!$-distribution matrix is $\!\begin{smallmx}\frac16&\frac13&\frac16\\0&\frac16&0\\0&\frac16&0\end{smallmx}\!$,
yielding the limit shape at the right of Figure~\ref{figLimitShape}.

For classes with more than one component, we state the following without proof:
Suppose $M$ has components $M_1,\ldots,M_k$.
We write $i\gg j$ if $\liminfty|\Grid_n(M_j)|/|\Grid_n(M_i)|=0$.
If there exists a unique $d$ such that $d\gg i$ for all $i\neq d$,
then,
in a uniformly chosen $M$\!-gridded permutation,
with high probability the proportion of points in the cells of the dominant component $M_d$ tends to 1.
In this case, $\mu_M=\mu_{M_d}$.
There are multiple dominant components, $M_1,\ldots,M_d$ say, if $i\gg j$ whenever $i\leqs d$ and $j>d$ but
we don't have $i\gg j$ for any $i,j\leqs d$.
In this case, $\Grid(M)$ has a \emph{random} limit shape $\Sigma_{i=1}^d\lambda_i\mu_{M_i}$, where $\lambda_1,\ldots,\lambda_d$ are chosen uniformly at random from the region of the hyperplane $\Sigma_{i=1}^d\lambda_i=1$ in which all $\lambda_i$ are nonnegative.

It seems likely that fluctuations about the limit shape of a grid class can be described by coupled Brownian motions, as in the case of \emph{square} permutations (permutations in which every point is a record)~\cite{BS2020}.
However, that is beyond our scope here.
See~\cite{BBFGMP2020,BBFGMP2022,BBFGP2018} for other recent work on the limit shapes of permutation classes,
and~\cite{BevanPermutonsReport} for some conjectures, the strongest of which is due to Justin Troyka:
If every pattern in $B$ is skew decomposable, then the scaling limit of $\av(B)$ is the increasing permuton.

\section{Connected classes with one corner}\label{sectLTX}

We say that a cell $C$ of a gridding matrix is a \emph{corner cell} or just a \emph{corner}, if there is both another non-blank cell in the same row as~$C$ and also another non-blank cell in the same column as~$C$.
A cell that is not a corner is a \emph{non-corner} cell.

For example, \!\gcthree{4}{1,-1,0,1}{0,0,-1}{0,-1,1}\! has three corners, one in the top row and two in the bottom row.

In this section we consider connected monotone grid classes that have a single corner.
These are either \textsf{L}-shaped, \textsf{T}-shaped or \textsf{X}-shaped.
See Figure~\ref{figLTX} for an example of each.
\textsf{L}-shaped and \textsf{T}-shaped classes come in four different orientations.

\begin{figure}[ht]
  \centering
  \gcfour5{1,1,1,-1,1}{-1}{1}{-1}
  \hspace{.5in}
  \gcfour4{-1,-1,1,-1}{0,0,-1}{0,0,1}{0,0,1}
  \hspace{.5in}
  \gcfour5{0,-1}{0,1}{1,1,1,-1,1}{0,1}
  \caption{An \textsf{L}-shaped class, a \textsf{T}-shaped class, and an \textsf{X}-shaped class}\label{figLTX}
\end{figure}

The row containing the corner we call the \emph{main row}, and
the column containing the corner we call the \emph{main column}.
For brevity, \emph{non-corner} cells in the main row are simply called \emph{row cells},
and non-corner cells in the main column are called \emph{column cells}.
Thus each \textsf{L}, \textsf{T} or \textsf{X}-shaped class consists of a corner cell, some row cells and some column cells.

Throughout this section
we assume
that $\Grid(M)$ is a connected one-corner class with dimensions $(r+1)\times(c+1)$.
Thus $M$ has
$r+c+1$ non-blank cells: the corner,
$r$~column cells, and
$c$~row cells.
For example, for the \textsf{L}-shaped and \textsf{X}-shaped classes in Figure~\ref{figLTX}, we have $r=3$ and $c=4$.

Our analysis of connected one-corner classes is as follows.
First, we apply Theorem~\ref{thmConnectedDistrib} to determine the asymptotic distribution of points between the cells (Section~\ref{sectLTXProportions}).
Then we use generating functions to establish the asymptotics of the number of gridded permutations in such a class (Section~\ref{sectLTXGridded}).
This is followed by an investigation of the ways in which points can dance between cells,
including the introduction of an appropriate notion for these classes of a constrained gridded permutation (Section~\ref{sectLTXDancing}).

We then introduce the different {corner types} that can appear in \textsf{L}, \textsf{T} and \textsf{X}-shaped classes,
with Theorem~\ref{thmLTXAsympt} giving the asymptotics for these classes in terms of the corner type.
This is then followed by a detailed examination of each of the corner types
(Sections~\ref{sectPeakCorners} to~\ref{sectCornerComparison}).
To conclude, we prove that almost all gridded permutations are constrained (Section~\ref{sectLTXConstrained}).
Finally, we briefly consider how our approach can be extended beyond connected one-corner classes (Section~\ref{sectBeyondLTX}).

\subsection{The distribution of points between cells}\label{sectLTXProportions}

For a connected one-corner class, the asymptotic distribution of points between the cells in a typical $M$\!-gridded permutation can easily be determined from Theorem~\ref{thmConnectedDistrib} and Proposition~\ref{propMaxGR}.
The asymptotic proportion of points in each row cell satisfies the same equations and hence these proportions are all equal.
The same is true for the column cells.

We use $\alpha$ to denote the proportion of points in the corner cell,
$\beta$ to denote the proportion in each of the row cells, and
$\gamma$ the proportion in each of the column cells.
So, for example,
the unique maximal distribution matrix for the
\textsf{X}-shaped class in Figure~\ref{figLTX} has the form
\[
\newcommand{\gz}{{\color{gray}0}}
\begin{smallmx}
\gz&\gamma &\gz&\gz&\gz \\
\gz&\gamma &\gz&\gz&\gz \\
\beta &\alpha &\beta &\beta &\beta \\
\gz&\gamma &\gz&\gz&\gz
\end{smallmx}
\!\!.
\]
Then, by Theorem~\ref{thmConnectedDistrib} and Proposition~\ref{propMaxGR}, we know that $\alpha$, $\beta$ and $\gamma$ are the unique positive solutions to the equations
\[
\alpha+c\beta+r\gamma=1
\quad \text{and} \quad
\frac{\alpha^2}{(\alpha+c\beta)(\alpha+r\gamma)} \;=\; \frac{\beta}{\alpha+c\beta} \;=\; \frac{\gamma}{\alpha+r\gamma}
.
\]
Solving these then yields
\begin{equation}\label{eqProportions}
\alpha \eq \frac1q ,
\qquad
\beta \eq \frac{c-r+q-1}{2 c q} ,
\qquad
\gamma \eq \frac{r-c+q-1}{2 r q}
,
\end{equation}
where
\begin{equation}\label{eqQ}
q \eq \sqrt{(r+c+1)^2-4 c r} .
\end{equation}
Note that
\begin{equation}\label{eqLambda}
\lambda \eq \frac{\beta}{\alpha+c\beta} \eq \frac{\gamma}{\alpha+r\gamma} \eq \frac{\alpha^2}{(\alpha+c\beta)(\alpha+r\gamma)} \eq \frac{r+c+1-q}{2 r c}
\end{equation}
is the common value of the ratios from Proposition~\ref{propMaxGR}.

For example, if $r=3$ and $c=4$ then we have $q=4$, $\alpha=\frac14$, $\beta=\frac18$, $\gamma=\frac1{12}$ and $\lambda=\frac16$.

\subsection{The asymptotics of gridded classes}\label{sectLTXGridded}

For skinny monotone grid classes, exact enumeration of the gridded permutations is simple:
a $k$-cell skinny grid class has exactly $k^n$ gridded $n$-permutations (Proposition~\ref{propSkinnyGriddedCount}).
For non-skinny classes things are not so straightforward.

Our approach is to determine the generating function for the gridded class, using a technique first described in~\cite[Chapter~4]{BevanThesis}.
We  then extract the asymptotic growth of the number of gridded permutations from the generating function using standard methods from analytic combinatorics.

To determine the generating function for the gridded permutations in an \textsf{L}-shaped, \textsf{T}-shaped or \textsf{X}-shaped class
we ``stitch together'' two skinny classes, one formed by the main row and the other by the main column.

The bivariate generating function for the $(c+1)$-cell horizontal skinny gridded class $\HHH^\#$ formed from the main row,
in which $x$ is used to mark the points in the corner cell, is
\[
H^\#(z,x) \eq \frac1{1-c\+z-z\+x} .
\]
Similarly, the bivariate generating function for the $(r+1)$-cell vertical skinny gridded class $\VVV^\#$ formed from the main column,
in which $y$ is used to mark the points in the corner cell, is
\[
V^\#(z,y) \eq \frac1{1-r\+z-z\+y} .
\]
Thus, the set of pairs $(\sigma_{\mathsf{h}}^\#, \sigma_{\mathsf{v}}^\#)$, consisting of
an $\HHH$-gridded permutation $\sigma_{\mathsf{h}}^\#$
and
a $\VVV$-gridded permutation $\sigma_{\mathsf{v}}^\#$,
is enumerated by the product of the generating functions of the two skinny classes:
\[
P^\#(z,x,y) \;=\;
H^\#(z,x)\+V^\#(z,y) \;=\; \frac1{(1-c\+z-z\+x)(1-r\+z-z\+y)} .
\]

To count $M$\!-gridded permutations, we are only interested in those pairs $(\sigma_{\mathsf{h}}^\#, \sigma_{\mathsf{v}}^\#)$
for which
the number of points of $\sigma_{\mathsf{h}}^\#$ in the corner cell is the same as
the number of points of $\sigma_{\mathsf{v}}^\#$ in the corner cell.
Any $M$\!-gridded permutation can be decomposed into such a pair,
and any such pair combine in a unique way to form an $M$\!-gridded permutation, as illustrated in Figure~\ref{figHVStitching}.

\begin{figure}[ht]
  \centering
  \raisebox{16pt}{
  \begin{tikzpicture}[scale=0.225]
  \fill[blue!10!white] (.5,.5) rectangle (9.5,9.5);
  \fill[green!25!white] (2.5,.5) rectangle (5.5,9.5);
  \plotgriddedperm{9}{1,7,4,5,8,9,2,6,3}{2,5,7}{}
  \circpt34 \circpt45 \circpt58
  \end{tikzpicture}
  ~~ \raisebox{26pt}{\:+} ~~
  \raisebox{3pt}{
  \begin{tikzpicture}[scale=0.225]
  \fill[yellow!30!white] (.5,.5) rectangle (8.5,8.5);
  \fill[green!25!white] (.5,3.5) rectangle (8.5,6.5);
  \plotgriddedperm{8}{8,4,1,5,2,3,7,6}{}{3,6}
  \circpt24 \circpt45 \circpt86
  \end{tikzpicture}
  }
  }
  ~~ \raisebox{42pt}{=\:\:} ~~
  \begin{tikzpicture}[scale=0.225]
  \fill[blue!10!white] (.5,3.5) rectangle (14.5,12.5);
  \fill[yellow!30!white] (2.5,.5) rectangle (10.5,14.5);
  \fill[green!25!white] (2.5,3.5) rectangle (10.5,12.5);
  \plotgriddedperm{14}{4,10,14,7,1,8,2,3,13,11,12,5,9,6}{2,10,12}{3,12}
  \circpt47 \circpt68 \circpt{10}{11}
  \end{tikzpicture}
  \caption{Stitching together a \!\gcone{4}{1,1,-1,-1}\!-gridded permutation and a \!\gcthree{1}{-1}{1}{1}\!-gridded permutation to
  create a \!\gcthree{4}{0,-1}{1,1,-1,-1}{0,1}\!-gridded permutation}\label{figHVStitching}
\end{figure}

In terms of the generating function $P^\#(z,x,y)$,
we need to extract the terms in which $x$ and $y$ have the same exponent.
However, we also need to correct for the double-counting of the points in the corner cell.
Thus we want
\[
\sum_{m\geqs0} \big[x^m y^m\big]P^\#(z,\,x/\!\sqrt{z},\,y/\!\sqrt{z}) .
\]
Here, the division of the second and third arguments by $\sqrt{z}$ decreases the exponent of $z$ by one for each point in the corner.

To extract the terms in which $x$ and $y$ have the same exponent, we let $y=x^{-1}$.
This yields a Laurent series\footnote{A \emph{Laurent series} is a power series in which terms of negative degree are permitted.} in $x$, where $x$ now records for each pair $(\sigma_{\mathsf{h}}^\#, \sigma_{\mathsf{v}}^\#)$ the difference between the number of points of $\sigma_{\mathsf{h}}^\#$ in the
corner
and the number of points of $\sigma_{\mathsf{v}}^\#$ in the
corner.
We just want the constant term (when this difference is zero):
\[
\big[x^0\big]P^\#(z,\,x/\!\sqrt{z},\,x^{-1}\!/\!\sqrt{z}) .
\]

To extract this constant term, we use the following result.
\begin{prop}[{Stanley~\cite[Section 6.3]{Stanley1999}}]\label{propDiagonalGF}
  If $f(x)=f(z,x)$ is a rational Laurent series in $x$, then the constant term $[x^0]f(x)$ is given by the sum of the residues\footnote{The residue of $h(x)$ at $x=\alpha$ is 
  the coefficient of $(x-\alpha)^{-1}$ in the Laurent expansion of $h(x)$ around $x=\alpha$. If $\alpha$ is a simple pole, then this is just the value of $y\+h(y+\alpha)$ at $y=0$.} of $x^{-1}f(x)$ at those 
  poles $\alpha(z)$ of $f(x)$ for which $\lim\limits_{z\rightarrow0}\alpha(z)=0$. These are known as the \emph{small} poles.
\end{prop}

In our case,
\[
x^{-1}\+P^\#(z,\,x/\!\sqrt{z},\,x^{-1}\!/\!\sqrt{z}) \eq \frac1{(1-c\+z-x\sqrt{z})(x-r\+z\+x-\sqrt{z})}
\]
has two poles, one at $x_1(z)=(1-c\+z)/\sqrt{z}$, and the other at $x_2(z)=\sqrt{z}/(1-r\+z)$.
Only $x_2(z)$ is small.
Thus all we need is the residue at $x=x_2(z)$.
This then yields the following generating function for $M$\!-gridded permutations in a connected class with one corner:
\[
F_M^\#(z)
\eq \sum_{n\geqs 0}|\Gridhash_n(M)| \, z^n
\eq \frac1{1-(r+c+1)z + r\+c\+z^2} .
\]

We  extract the asymptotic growth of the number of gridded permutations from this generating function by using the following standard result.
\begin{prop}[{see \cite[Theorems IV.10 and VI.1]{FS2009}}]
  Suppose $F(z)$ is the ordinary generating function of a combinatorial class $\CCC$.
  Let $\rho$ be the least singularity of $F(z)$ on the positive real axis.
  If there are no other singularities on the radius of convergence and $\rho$ is a pole of order~$r$, then
  \[
  |\CCC_n| \;\sim\; c{\rho^{-n}}\+n^{r-1}
  \qquad \text{where} \quad
  c \;=\; \dfrac{\rho^{-r}}{(r-1)!} \+ \lim\limits_{z\to\rho} \+ (\rho-z)^r\+F(z) .
  \]
\end{prop}
In the case that $F(z)$ is a rational function with a denominator $Q(z)$ of degree $d$, the exponential growth rate $\rho^{-1}$ is the greatest root of the polynomial $z^dQ(z^{-1})$.
Thus the growth rate of $\Grid^\#(M)$, and hence also of $\Grid(M)$, equals the larger of the two roots of the quadratic equation $z^2 - (r+c+1)z+rc=0$:
\[
g_M \eq
\gr(\Grid(M)) \eq \tfrac12 (r+c+1+q),
\]
where $q$ is defined in equation~\eqref{eqQ} on page~\pageref{eqQ}.
From~\cite{AV2019}, we know that $g_M=g(\Gamma_M)$, so $g_M$ could also have been calculated from the proportions in equation~\eqref{eqProportions} by using Proposition~\ref{propCountnGammaGriddings}.

Since $z=g_M^{-1}$ is a simple pole of $F_M^\#(z)$, the subexponential term is just a constant:
\[
\theta_M^\#(n) \eq (r+c+1+q)/2q \eq g_M/q .
\]
Thus the asymptotic growth of the number of gridded permutations in a connected one-corner class is given by the following proposition.
\begin{prop}\label{propGriddedAsymptoticsLTX}
  If $M$ is connected with one corner and has dimensions $(r+1)\times(c+1)$, then
\[
\big|\Gridhash_n(M)\big| \;\sim\; \theta^\#g^n,
\quad
\text{where~~}
\theta^\# \eq \frac{r+c+q+1}{2q}
\quad
\text{and~~}
g \eq \frac{r+c+q+1}2 ,
\]
with $q \eq \sqrt{(r+c+1)^2-4 c r}$.
\end{prop}
For example, if $M$ has dimensions $4\times5$ ($r=3$ and $c=4$), then $\big|\Gridhash_n(M)\big|\sim \frac32\times6^n$.

\subsection{Dancing: peaks, diagonals and tees}\label{sectLTXDancing}

Generalising from skinny classes, we say that two non-blank cells that are adjacent (either horizontally or vertically) in a gridding matrix form a \emph{peak} if one is increasing (\cUp) and the other is decreasing~(\cDown).
Thus, in addition to peaks that point up or down, we also have peaks that \emph{point left} ($\!\gctwo{1}{1}{-1}\!$) and peaks that \emph{point right}~($\!\gctwo{1}{-1}{1}\!$).

Given a gridded permutation $\sigma^\#$ with at least two points in each cell in a gridded class with a peak $\Lambda$ that points left or right, the peak point of $\Lambda$ is the leftmost of the points of $\sigma^\#$ in the two cells of $\Lambda$ if $\Lambda$ points left, and is the rightmost of the points of $\sigma^\#$ in the two cells of $\Lambda$ if $\Lambda$ points right.
For example, the rightmost circled point in the gridded permutation at the right of Figure~\ref{figHVStitching} on page~\pageref{figHVStitching} is the peak point of a peak that points right.

Generalising the notion of dancing we introduced in Section~\ref{sectDancing},
if $Q$ is a peak point 
adjacent to the row or column divider which separates the two cells of the peak,
then $Q$ can dance.
The movement of this divider to the other side of $Q$ results in another valid gridding.

For one-corner classes, in addition to peaks, we also need to take into account two other structures, both of which contain
a diagonally adjacent pair of cells oriented either \longUp{} or~\longDown.
In one-corner classes, these can only occur adjacent to the corner.

If the corner is oriented in the same direction as its two neighbours, then they form a \emph{diagonal}, seen in one of the four rotations of~\longUpB.
For example, the \textsf{X}-shaped class at the right of Figure~\ref{figLTX} on page~\pageref{figLTX} has two diagonals.

On the other hand, if the corner is oriented in the opposite direction from its two neighbours, then the three cells form a \emph{tee}: one of the four rotations of~\tee{}.
For example, the \textsf{T}-shaped class in the centre of Figure~\ref{figLTX} has a tee.

We say that a peak is a \emph{corner peak} if one of its two cells is a corner cell,
and that it is a \emph{non-corner peak} otherwise.
For example, the \textsf{L}-shaped class at the left of Figure~\ref{figLTX} has one corner peak and four non-corner peaks.
In connected one-corner classes, 
diagonals and tees only occur at corners:
diagonals can only be formed by two cells which are adjacent to a corner,
and tees can only be formed by a corner cell and two of its neighbours.

The presence of a diagonal or tee makes possible new ways for points to dance.
The dancing of a peak point, which we now call \emph{peak dancing} or \emph{dancing at a peak}, involves only a single point.
In contrast, \emph{diagonal dancing} and \emph{tee dancing} may involve more than one point.

\begin{figure}[ht]
  \centering
  \begin{tikzpicture}[scale=0.225] \plotgriddedperm{9}{5,1,8,9,2,3,4,6,7}{4}{1} \circpt{5}{2}\circpt{6}{3}\circpt{7}{4}
  \end{tikzpicture} ~~~~~~~~
  \begin{tikzpicture}[scale=0.225] \plotgriddedperm{9}{5,1,8,9,2,3,4,6,7}{5}{2} \circpt{5}{2}\circpt{6}{3}\circpt{7}{4}
  \end{tikzpicture} ~~~~~~~~
  \begin{tikzpicture}[scale=0.225] \plotgriddedperm{9}{5,1,8,9,2,3,4,6,7}{6}{3} \circpt{5}{2}\circpt{6}{3}\circpt{7}{4}
  \end{tikzpicture} ~~~~~~~~
  \begin{tikzpicture}[scale=0.225] \plotgriddedperm{9}{5,1,8,9,2,3,4,6,7}{7}{4} \circpt{5}{2}\circpt{6}{3}\circpt{7}{4}
  \end{tikzpicture}
  \caption{The four \!\gctwo{2}{1,1}{1}\!-griddings of 518923467; 
  the three circled points can dance diagonally}\label{figDiagonalDancing}
\end{figure}

Suppose that $M$ is a gridding matrix, with a diagonal formed from two cells
separated by a row divider and a column divider.
Suppose that $\sigma^\#$ is an $M$\!-gridded permutation, and that $Q$ is a point of $\sigma^\#$
in
one of the two cells of
the diagonal.
If the simultaneous movement of both dividers immediately to the other
side of $Q$ results in another valid $M$\!-gridding of~$\sigma$,
then we say that $Q$ can \emph{dance diagonally}.
See Figure~\ref{figDiagonalDancing} for an illustration.
Note that any points that lie both horizontally and vertically between $Q$ and the intersection of the dividers can also dance.
This yields a monotone sequence of points all of which are able to dance.

Similarly, suppose that $M$ is a gridding matrix, with a tee formed from a corner cell and two other cells, $C_1$ and $C_2$ say,
separated by a row divider and a column divider.
Suppose that $\sigma^\#$ is an $M$\!-gridded permutation, and that $Q$ is a point of $\sigma^\#$
in
one of the three cells of
the tee.
If there is an alternating sequence of one-step moves of the two dividers
which leaves one of them immediately the other side of $Q$
and results in another valid $M$\!-gridding of~$\sigma$,
then we say that $Q$ can \emph{dance through the tee}.
See Figure~\ref{figTeeDancing} for an illustration.

\begin{figure}[ht]
  \centering
  \begin{tikzpicture}[scale=0.225] \plotgriddedperm{9}{8,1,6,2,3,5,7,4,9}{3,7}{1} \circpt{4}{2}\circpt{5}{3}\circpt{6}{5}
  \end{tikzpicture} ~~~
  \begin{tikzpicture}[scale=0.225] \plotgriddedperm{9}{8,1,6,2,3,5,7,4,9}{4,7}{1} \circpt{4}{2}\circpt{5}{3}\circpt{6}{5}
  \end{tikzpicture} ~~~
  \begin{tikzpicture}[scale=0.225] \plotgriddedperm{9}{8,1,6,2,3,5,7,4,9}{4,7}{2} \circpt{4}{2}\circpt{5}{3}\circpt{6}{5}
  \end{tikzpicture} ~~~
  \begin{tikzpicture}[scale=0.225] \plotgriddedperm{9}{8,1,6,2,3,5,7,4,9}{5,7}{2} \circpt{4}{2}\circpt{5}{3}\circpt{6}{5}
  \end{tikzpicture} ~~~
  \begin{tikzpicture}[scale=0.225] \plotgriddedperm{9}{8,1,6,2,3,5,7,4,9}{5,7}{3} \circpt{4}{2}\circpt{5}{3}\circpt{6}{5}
  \end{tikzpicture} ~~~
  \begin{tikzpicture}[scale=0.225] \plotgriddedperm{9}{8,1,6,2,3,5,7,4,9}{6,7}{3} \circpt{4}{2}\circpt{5}{3}\circpt{6}{5}
  \end{tikzpicture}
  \caption{The six \!\gctwo{3}{-1,1,1}{1}\!-griddings of 816235749; 
  the three circled points can dance through the tee}\label{figTeeDancing}
\end{figure}

\label{defDiagDancing}
Again, any points that lie between $Q$ and the intersection of the dividers can also dance,
yielding a monotone sequence of points that can dance.
Note, however, that the first (or last) point in this sequence, lying in $C_1$ say, may only be able to dance into the corner, and not be able to dance through to $C_2$.
For example, the rightmost circled point in the gridded permutations in Figure~\ref{figTeeDancing} can't dance into the cell below the corner.

For almost all permutations in the classes we consider,
the valid griddings are restricted to those that can be obtained through
peak dancing, diagonal dancing and tee dancing.
So, generalising the definition for skinny classes, if $\Grid(M)$ is a connected one-corner class
we say that an $M$\!-gridded permutation $\sigma^\#$ is \emph{$M$\!-constrained} (or just \emph{constrained}) if
\begin{itemize}
  \item[(a)] every $M$\!-gridding of its underlying permutation $\sigma$ is the result of zero or more points of $\sigma^\#$ dancing at a peak or diagonally or through a tee,
  and
  \item[(b)] in every $M$\!-gridding of $\sigma$, each non-blank cell contains at least two points.
\end{itemize}
We defer further analysis and a proof that most $M$\!-gridded permutations are constrained until Section~\ref{sectLTXConstrained}, after a discussion of the different corner types.

\subsection{Counting griddings}

To complete our computation of the asymptotic growth of
connected one-corner
classes,
we determine how a constrained $M$\!-gridded permutation $\sigma^\#$ must be structured so that its underlying permutation $\sigma$ has a given number of griddings.
This depends on the orientation (either increasing or decreasing) of the corner cell and the orientation of each of its non-blank neighbours.

In \textsf{L}-shaped classes, there are 8 distinct ways to orient the corner cell and the two non-blank cells adjacent to it.
In \textsf{T}-shaped classes, there are 16 distinct ways to orient the corner cell and the three non-blank cells adjacent to it.
And, in \textsf{X}-shaped classes, there are 32 distinct ways to orient the corner cell and the four non-blank cells adjacent to it.
These are all illustrated in Figure~\ref{figCornerTypes}.
We call these the \emph{corner types}.

The subscripts in the names used for corner types are determined by reading the non-blank cells in normal reading order (left-to-right and top-to-bottom) and treating \cDown{} as 0 and \cUp{} as 1 to give a binary number.
For example, $T_5$ is \gctwo3{-1,1,-1}{0,1} since $5=0101_2$.

Two corner types, $\tau_1$ and $\tau_2$, are \emph{equivalent} (denoted $\tau_1\cong\tau_2$ in Figure~\ref{figCornerTypes}) if $\tau_2$ can be obtained from $\tau_1$ by rotation or reflection and/or by the addition or deletion of non-blank cells without creating or removing any peaks, diagonals or tees.

For example, the following corner types are equivalent:
\[
  \genfrac{}{}{0pt}0{\gctwo2{-1,-1}{1}}{L_1}
  \; \raisebox{6pt}{${}\cong{}$} \;
  \genfrac{}{}{0pt}0{\gctwo3{-1,-1,-1}{0,1}}{T_1}
  \; \raisebox{6pt}{${}\cong{}$} \;
  \genfrac{}{}{0pt}0{\gctwo3{1,1,1}{0,-1}}{T_{14}}
  \; \raisebox{6pt}{${}\cong{}$} \;
  \genfrac{}{}{0pt}0{\gctwo2{1,1}{-1}}{L_6}
  \; \raisebox{6pt}{${}\cong{}$} \;
  \genfrac{}{}{0pt}0{\gctwo2{1,-1}{1}}{L_5} .
\]
Specifically,
$L_1\cong T_1$ by the addition of a decreasing cell at the left,
$T_1\cong T_{14}$ by reflection about a vertical axis,
$T_{14}\cong L_6$ by the deletion of an increasing cell from the left, and
$L_6\cong L_5$ by reflection about a diagonal axis.

The analysis of two equivalent corner types is the same.
As seen from Figure~\ref{figCornerTypes}, there are eleven inequivalent corner types to consider.
Of equivalent corner types, we choose the one with the least subscript as the representative.
These are also shown in Table~\ref{tblKappa} on page~\pageref{tblKappa} below.

\begin{figure}[t]
\newcommand{\lshapeA}[3]{\centering$\genfrac{}{}{0pt}0{\gctwo2{#1}{#2}}{L_{#3}}$}
\newcommand{\lshapeB}[4]{\centering$\genfrac{}{}{0pt}0{\color{gray!75!black}\gctwo2{#1}{#2}}{\!{\color{gray!75!black}L_{#3}\cong{}#4}\!}$}
\newcommand{\tshapeA}[3]{\centering$\genfrac{}{}{0pt}0{\gctwo3{#1}{#2}}{T_{#3}}$}
\newcommand{\tshapeB}[4]{\centering$\genfrac{}{}{0pt}0{\color{gray!75!black}\gctwo3{#1}{#2}}{\!{\color{gray!75!black}T_{#3}\cong{}#4}\!}$}
\newcommand{\xshapeA}[4]{\centering$\genfrac{}{}{0pt}0{\gcthree3{#1}{#2}{#3}}{X_{#4}}$}
\newcommand{\xshapeB}[5]{\centering$\genfrac{}{}{0pt}0{\color{gray!75!black}\gcthree3{#1}{#2}{#3}}{\!\!{\color{gray!75!black}X_{#4}\cong{}#5}\!\!}$}
\begin{center}
\begin{tabular}{|p{40pt}|p{40pt}|p{40pt}|p{40pt}|p{40pt}|p{40pt}|p{40pt}|p{40pt}|}
  \lshapeA{-1,-1}{-1}0 %
& \lshapeA{-1,-1}{1}1 %
& \lshapeB{-1,1}{-1}2{L_1}
& \lshapeA{-1,1}{1}3 %
& \lshapeA{1,-1}{-1}4 %
& \lshapeB{1,-1}{1}5{L_1}
& \lshapeB{1,1}{-1}6{L_1}
& \lshapeA{1,1}{1}7 %
\end{tabular}
\end{center}

\begin{center}
\begin{tabular}{|p{40pt}|p{40pt}|p{40pt}|p{40pt}|p{40pt}|p{40pt}|p{40pt}|p{40pt}|}
  \tshapeB{-1,-1,-1}{0,-1}0{L_7}
& \tshapeB{-1,-1,-1}{0,1}1{L_1}
& \tshapeA{-1,-1,1}{0,-1}2 %
& \tshapeB{-1,-1,1}{0,1}3{L_3}
& \tshapeA{-1,1,-1}{0,-1}4 %
& \tshapeA{-1,1,-1}{0,1}5 %
& \tshapeB{-1,1,1}{0,-1}6{L_3}
& \tshapeB{-1,1,1}{0,1}7{T_2}
\end{tabular}%
\vspace{6pt}

\begin{tabular}{|p{40pt}|p{40pt}|p{40pt}|p{40pt}|p{40pt}|p{40pt}|p{40pt}|p{40pt}|}
  \tshapeB{1,-1,-1}{0,-1}8{L_1}
& \tshapeB{1,-1,-1}{0,1}9{L_4}
& \tshapeB{1,-1,1}{0,-1}{10}{T_5}
& \tshapeB{1,-1,1}{0,1}{11}{T_4}
& \tshapeB{1,1,-1}{0,-1}{12}{L_4}
& \tshapeB{1,1,-1}{0,1}{13}{L_1}
& \tshapeB{1,1,1}{0,-1}{14}{L_1}
& \tshapeB{1,1,1}{0,1}{15}{L_7}
\end{tabular}
\end{center}

\begin{center}
\begin{tabular}{|p{40pt}|p{40pt}|p{40pt}|p{40pt}|p{40pt}|p{40pt}|p{40pt}|p{40pt}|}
  \xshapeA{0,-1}{-1,-1,-1}{0,-1}0 %
& \xshapeB{0,-1}{-1,-1,-1}{0,1}1{L_1}
& \xshapeB{0,-1}{-1,-1,1}{0,-1}2{L_1}
& \xshapeB{0,-1}{-1,-1,1}{0,1}3{L_3}
& \xshapeA{0,-1}{-1,1,-1}{0,-1}4 %
& \xshapeB{0,-1}{-1,1,-1}{0,1}5{T_4}
& \xshapeB{0,-1}{-1,1,1}{0,-1}6{T_4}
& \xshapeA{0,-1}{-1,1,1}{0,1}7 %
\end{tabular}%
\vspace{6pt}

\begin{tabular}{|p{40pt}|p{40pt}|p{40pt}|p{40pt}|p{40pt}|p{40pt}|p{40pt}|p{40pt}|}
  \xshapeB{0,-1}{1,-1,-1}{0,-1}8{L_1}
& \xshapeB{0,-1}{1,-1,-1}{0,1}9{X_7}
& \xshapeB{0,-1}{1,-1,1}{0,-1}{10}{T_5}
& \xshapeB{0,-1}{1,-1,1}{0,1}{11}{T_4}
& \xshapeB{0,-1}{1,1,-1}{0,-1}{12}{T_4}
& \xshapeB{0,-1}{1,1,-1}{0,1}{13}{L_3}
& \xshapeB{0,-1}{1,1,1}{0,-1}{14}{T_5}
& \xshapeB{0,-1}{1,1,1}{0,1}{15}{L_1}
\end{tabular}%
\vspace{6pt}

\begin{tabular}{|p{40pt}|p{40pt}|p{40pt}|p{40pt}|p{40pt}|p{40pt}|p{40pt}|p{40pt}|}
  \xshapeB{0,1}{-1,-1,-1}{0,-1}{16}{L_1}
& \xshapeB{0,1}{-1,-1,-1}{0,1}{17}{T_5}
& \xshapeB{0,1}{-1,-1,1}{0,-1}{18}{X_7}
& \xshapeB{0,1}{-1,-1,1}{0,1}{19}{T_4}
& \xshapeB{0,1}{-1,1,-1}{0,-1}{20}{T_4}
& \xshapeB{0,1}{-1,1,-1}{0,1}{21}{T_5}
& \xshapeB{0,1}{-1,1,1}{0,-1}{22}{L_3}
& \xshapeB{0,1}{-1,1,1}{0,1}{23}{L_1}
\end{tabular}%
\vspace{6pt}

\begin{tabular}{|p{40pt}|p{40pt}|p{40pt}|p{40pt}|p{40pt}|p{40pt}|p{40pt}|p{40pt}|}
  \xshapeB{0,1}{1,-1,-1}{0,-1}{24}{L_3}
& \xshapeB{0,1}{1,-1,-1}{0,1}{25}{T_4}
& \xshapeB{0,1}{1,-1,1}{0,-1}{26}{T_4}
& \xshapeB{0,1}{1,-1,1}{0,1}{27}{X_4}
& \xshapeB{0,1}{1,1,-1}{0,-1}{28}{X_7}
& \xshapeB{0,1}{1,1,-1}{0,1}{29}{L_1}
& \xshapeB{0,1}{1,1,1}{0,-1}{30}{L_1}
& \xshapeB{0,1}{1,1,1}{0,1}{31}{X_0}
\end{tabular}

\end{center}
  \caption{The corner types}\label{figCornerTypes}
\end{figure}


Let $\bshn$ be drawn uniformly at random from $\Gridhash_n(M)$ and $\bsn$ be its underlying permutation.
By combining the structural analysis above with the asymptotic distribution of points between cells from Section~\ref{sectLTXProportions},
we then calculate, for each $\ell\geqs1$, the asymptotic probability
\[
P_\ell \eq \liminfty \prob{\text{$\bsn$ has exactly $\ell$ distinct $M$\!-griddings}} .
\]
Then, letting
\[
\kappa_M
\eq \sum_{\ell\geqs1}P_\ell/\ell 
\eq \liminfty \frac{\big|\Grid_n(M)|}{\big|\Gridhash_n(M)|}
\]
be the \emph{correction factor} for the class,
we conclude that
$
\big|\Grid_n(M)\big|
\sim
\kappa_M \, \theta^\# \, g^n ,
$
where $\theta^\#$ and $g$ are given by Proposition~\ref{propGriddedAsymptoticsLTX}.

To determine the probabilities, we make repeated use of the following observation, recalling the definition of the common ratio $\lambda$ in equation~\eqref{eqLambda} on page~\pageref{eqLambda}.
\begin{obs}\label{obsProbs}
  Suppose $\Grid(M)$ is a connected one-corner class with dimensions $(r+1)\times(c+1)$, and
  let $\alpha$, $\beta$ and $\gamma$ be the asymptotic proportion of points of an $M$\!-gridded permutation
  in the corner cell, in any row cell, and in any column cell, respectively.
  Let $\lambda=\beta/(\alpha+c\beta)=\gamma/(\alpha+r\gamma)$.
  Suppose $\bshn$ is an $M$\!-gridded permutation chosen uniformly at random.
  Then, for each $k\geqs1$, we have
  \[
  \liminfty
  \prob{\text{the $k$th point from the top in the main row of $\bshn$ occurs in a given row cell}}
  \eq
  \lambda ,
  \]
  and similarly for the $k$th point from the bottom,
  and for the occurrence in a given column cell of the $k$th point from the left or right in the main column.

  Moreover, for $j\neq k$, the events that the $j$th and $k$th points from the top or bottom 
  in the main row 
  occur in specific cells are asymptotically independent.
  And analogously for points in the main column.
\end{obs}
\begin{proof}
  Let $A$ be the event that the $k$th point from the top in the main row of $\bshn$ occurs in the given row cell.
  Let $L$ be the number of points in the given cell and $M$ be the total number of points in the main row.
  Then, conditioning on these numbers, we have
  \[
  \prob{A\:\big|\:L=\ell \:\wedge\: M=m} \eq \binom{m-1}{\ell-1}\Big/\binom{m}{\ell} \eq \frac{\ell}m .
  \]
  Now by Theorem~\ref{thmConnectedDistrib}, for any $\veps>0$, we have
  \[
  \liminfty\prob{\big|L-\beta n\big|\leqs\veps n}\eq1
  \qquad \text{and} \qquad
  \liminfty\prob{\big|M-\beta n/\lambda\big|\leqs\veps n}\eq1
  .
  \]
  Hence, $\liminfty\prob{A}=\lambda$.

  Now let $B$ be the event that the $j$th and $k$th points from the top in the main row of $\bshn$ both occur in a given row cell.
  Then,
  \[
  \prob{B\:\big|\:L=\ell \:\wedge\: M=m} \eq \binom{m-2}{\ell-2}\Big/\binom{m}{\ell}
  \eq \frac{\ell(\ell-1)}{m(m-1)} .
  \]
  Applying Theorem~\ref{thmConnectedDistrib} then yields $\liminfty\prob{B}=\lambda^2$, as required.

  A similar argument handles the event that the $j$th and $k$th points occur in two specific distinct cells.
\end{proof}

Note also that if dancing can occur asymptotically independently in more than one location, then the corresponding correction factors multiply.
This follows from the following arithmetic observation.
\begin{obs}\label{obsIndepDancing}
  Given real sequences $(P'_i)_{i\geqs1}$ and $(P''_j)_{j\geqs1}$,
  let $\kappa'=\sum\limits_{i\geqs1}P'_i/i$ and $\kappa''=\sum\limits_{j\geqs1}P''_j/j$.
  Suppose $P_\ell =\sum\limits_{ij=\ell }P'_i\,P''_j$.
  Then,
  \[
  \sum\limits_{\ell \geqs1}P_\ell /\ell  \eq \kappa'\,\kappa'' .
  \]
\end{obs}
\begin{proof}
  \[
  \kappa'\,\kappa''
  \eq
  \sum\limits_{i\geqs1}\frac{P'_i}i \, \sum\limits_{j\geqs1}\frac{P''_j}j
  \eq
  \sum\limits_{i,j\geqs1} \frac{P'_i\,P''_j}{ij}
  \eq
  \sum\limits_{\ell \geqs1} \frac1\ell  \sum\limits_{ij=\ell } P'_i\,P''_j
  \eq
  \sum\limits_{\ell \geqs1} \frac{P_\ell }\ell
  . \qedhere
  \]
\end{proof}

\subsubsection*{\normalsize Non-corner peaks}

Before looking at each corner type in turn, we consider the effect of non-corner peaks.
By definition, in any constrained $M$\!-gridded permutation,
at each non-corner peak of $M$
there is a peak point which can dance.
So, if $M$ has $p$ non-corner peaks, these contribute a factor of $2^p$ to the number of possible $M$\!-griddings, in an analogous manner to skinny classes (Proposition~\ref{propSkinnyGriddings}).
Thus, since almost all $M$\!-gridded permutations are $M$\!-constrained
(see Proposition~\ref{propLTXConstrainedAreGeneric} below),
the asymptotic enumeration of \textsf{L}, \textsf{T} and \textsf{X}-shaped classes is given by the following result.

\begin{thm}\label{thmLTXAsympt}
  Suppose $\Grid(M)$ is a connected one-corner class with corner type $\tau$, and $p$ non-corner peaks, then
  \[
  \big|\Grid_n(M)\big|
  \;\sim\;
  2^{-p} \,
  \kappa(\tau) \,
  \theta^\# \,
  g^n ,
  \]
where $\kappa(\tau)$ is the correction factor for a gridding matrix with the same corner type and dimensions as $M$ but with no non-corner peaks,
and $\theta^\#$ and $g$ are as given by Proposition~\ref{propGriddedAsymptoticsLTX}.
\end{thm}

Note that the correction factor for the class is given by $\kappa_M=2^{-p}\, \kappa(\tau)$.

\subsubsection*{\normalsize Worked example}

Before turning to the calculation of the correction factors for each corner type, we very briefly illustrate our method 
with an example, by determining the asymptotic enumeration of the 
\mbox{\textsf{X}-shaped} class from Figure~\ref{figLTX}:
\[
  M_{\mathsf{X}} \eq
  \gcfour5{0,-1}{0,1}{1,1,1,-1,1}{0,1}
  \!.
\]
%
%
Firstly, from Proposition~\ref{propGriddedAsymptoticsLTX}, since $r=3$ and $c=4$, we have $\big|\Gridhash_n(M_{\mathsf{X}})\big|\sim \frac32\times6^n$.

Secondly, $M_{\mathsf{X}}$ has three non-corner peaks and has corner type $X_0$. 
Now, $\kappa(X_0)=(1-\lambda)^2$, and from equation~\eqref{eqLambda} we know that $\lambda =\frac16$.

Thus, by Theorem~\ref{thmLTXAsympt}, we have
\[
\big|\Grid_n(M_{\mathsf{X}})\big| \;\sim\; 2^{-3} \times \big(\tfrac56\big)^2 \times \tfrac32 \times 6^n \eq \tfrac{25}{192} \times 6^n.
\]

\subsubsection*{\normalsize Correction factors for corner types}

We now calculate
$\kappa(\tau)$ for each of the eleven inequivalent corner types.
Table~\ref{tblKappa} summarises our results.
In the analysis of a corner type, what is important is the number of peaks, diagonals and tees it contains, and how these are combined.

\begin{table}[t] 
  \centering
  \renewcommand{\arraystretch}{1.7}
  \setgcvpadding{.4}
  \begin{tabular}{|c|c|c|c|c|c|c|}
    \hline
    $\tau$ & & \small$P$ & \small$D$ & \small$T$ & $\kappa(\tau)$ & $\kappa(\rot{\tau})$ \\\hline\hline
    $L_0$ & \raisebox{-2.5pt}{\gctwo2{-1,-1}{-1}} & \small0 & \small0 & \small0 &
        \multicolumn{2}{c|}{$1$}  \\\hline
    $L_1$ & \raisebox{-2.5pt}{\gctwo2{-1,-1}{1}} & \small1 & \small0 & \small0 &
         ~ \raisebox{1.75pt}{$\frac12 \big(1+\frac{c \alpha \lambda}{\alpha + \gamma}\big)$} ~  &  ~ \raisebox{1.75pt}{$\frac12 \big(1+\frac{r \alpha \lambda}{\alpha + \beta}\big)$} ~  \\\hline
    $L_3$ & \raisebox{-2.5pt}{\gctwo2{-1,1}{1}} & \small0 & \small0 & \small1 &
        \multicolumn{2}{c|}{\raisebox{1.75pt}{$\frac{\lambda\,(1-\lambda)}{(1-(c-1)\lambda)\, (1-(r-1)\lambda)}$}}  \\\hline
    $L_4$ & \raisebox{-2.5pt}{\gctwo2{1,-1}{-1}} & \small2 & \small0 & \small0 &
        \multicolumn{2}{c|}{\raisebox{1pt}{$\kappa(L_1) \, \kappa(\rot{L_1})$}}  \\\hline
    $L_7$ & \raisebox{-2.5pt}{\gctwo2{1,1}{1}} & \small0 & \small1 & \small0 &
        \multicolumn{2}{c|}{$1-\lambda$}  \\\hline
    $T_2$ & \raisebox{-2.5pt}{\gctwo3{-1,-1,1}{0,-1}} & \small1 & \small1 & \small0 &
        \raisebox{1pt}{$\kappa(\rot{L_1}) \, \kappa(L_7)$} & \raisebox{1pt}{$\kappa(L_1) \, \kappa(L_7)$} \\\hline
    $T_4$ & \raisebox{-2.5pt}{\gctwo3{-1,1,-1}{0,-1}} & \small1 & \small0 & \small1 &
        \raisebox{1pt}{$\kappa(\rot{L_1}) \, \kappa(L_3)$} & \raisebox{1pt}{$\kappa(L_1) \, \kappa(L_3)$} \\\hline
    $T_5$ & \raisebox{-2.5pt}{\gctwo3{-1,1,-1}{0,1}} & \small2 & \small0 & \small0 &
        \raisebox{1pt}{$\kappa(\rot{L_1})^2$} & \raisebox{1pt}{$\kappa(L_1)^2$}  \\\hline
    $X_0$ & \raisebox{-2.5pt}{\gcthree3{0,-1}{-1,-1,-1}{0,-1}} & \small0 & \small2 & \small0 &
        \multicolumn{2}{c|}{$\kappa(L_7)^2$}  \\\hline
    $X_4$ & \raisebox{-2.5pt}{\gcthree3{0,-1}{-1,1,-1}{0,-1}} & \small0 & \small0 & \small2 &
        \multicolumn{2}{c|}{$\kappa(L_3)^2$}  \\\hline
    $X_7$ & \raisebox{-2.5pt}{\gcthree3{0,-1}{-1,1,1}{0,1}} & \small2 & \small1 & \small0 &
        \multicolumn{2}{c|}{$\kappa(L_1) \, \kappa(\rot{L_1}) \,\kappa(L_7)$}  \\\hline
  \end{tabular}
  \caption{The correction factors and the number of peaks ($P$), diagonals ($D$) and tees ($T$) for each of the eleven inequivalent corner types}\label{tblKappa}
\end{table}

Rotating a class by $90^\circ$ switches the roles of $r$ and~$c$ and of $\beta$ and~$\gamma$. We use $\rot{\tau}$ to denote a $90^\circ$ rotation of corner type~$\tau$.
In the figures used below for each corner type, dots indicate where additional row and column cells may occur, as long as they don't create additional corner peaks, diagonals or tees.

We begin with the simplest of the corner types.

\subsubsection*{\normalsize Corner type \texorpdfstring{$L_0$}{L0}}

\begin{center}
{\setgcptsize{.125}\setgcextra{\gchdots{2}{2}\gcvdots{0}{0}}\gctwo2{-1,-1}{-1}}
\end{center}

This corner type has no peaks, diagonals or tees.
It can't occur in a \textsf{T}-shaped or \textsf{X}-shaped class.
In a constrained gridded permutation, no dancing is possible, so the underlying permutation has a single gridding.
Thus $P_1=1$, and $P_\ell=0$ if $\ell>1$.
Hence $\kappa(L_0)=1$.

\subsection{Corners with peaks}\label{sectPeakCorners}

\subsubsection*{\normalsize Corner type \texorpdfstring{$L_1$}{L1}}

\begin{center}
{\setgcptsize{.125}
\setgcextra{\gchdots{2}{2}\gcvdots{0}{0}}\gctwo2{-1,-1}{1} \quad
\setgcextra{\gchdots{-2}{2}\gcvdots{1}{0}}\gctwo2{-1,-1}{0,1} \quad
\setgcextra{\gchdots{-2}{2}\gchdots{3}{2}\gcvdots{1}{0}}\gctwo3{-1,-1,-1}{0,1} \quad
\raisebox{3.67pt}{\setgcextra{\gchdots{-2}{2}\gcvdots{1}{0}\gcvdots{1}{5}}\gcthree2{0,-1}{-1,-1}{0,1}}}
\end{center}

This corner type has one peak.
It can't occur in an \textsf{X}-shaped class.
As can be seen from the figures, the peak may be orientated in different ways with respect to the other cells adjacent to the corner.
However, the analysis is the same in every case, so we consider this to be a single corner type.

Given a constrained gridded permutation,
let $Q$ be the peak point, and let $R$ be the lowest point in any of the row cells.
$Q$~is the only point that may be able to dance.
It can't dance if it is above $R$ in the corner cell.
Otherwise (if it is below $R$, either in the corner cell or in the cell below the corner) it can dance.
We say that point $R$ is the \emph{controller}, since it controls whether $Q$ can dance or not.
See Figure~\ref{figL1} for an illustration of the three cases.

\begin{figure}[ht]
  \centering
  \begin{tikzpicture}[scale=0.25]
    \plotgriddedperm{9}{7,5,1,2,4,9,6,8,3}{5,7}{2}
    \node at (4.1,4.9) {\footnotesize$Q$};
    \node at (8.2,3.9) {\footnotesize$R$};
  \end{tikzpicture}
  \qquad\qquad
  \begin{tikzpicture}[scale=0.25]
    \plotgriddedperm{9}{7,5,1,2,3,9,6,8,4}{5,7}{2}
    \circpt53
    \node at (4.1,3.9) {\footnotesize$Q$};
    \node at (8.2,4.9) {\footnotesize$R$};
  \end{tikzpicture}
  \qquad\qquad
  \begin{tikzpicture}[scale=0.25]
    \plotgriddedperm{9}{1,2,7,5,3,9,6,8,4}{5,7}{3}
    \circpt53
    \node at (4.1,2.1) {\footnotesize$Q$};
    \node at (8.2,4.9) {\footnotesize$R$};
  \end{tikzpicture}
  \caption{Three \!\gctwo{3}{-1,-1,-1}{1}\!-gridded permutations; the peak point is circled if it can dance}\label{figL1}
\end{figure}

The peak point, $Q$, is either in the corner or in the cell immediately below the corner.
Thus the asymptotic probability that $Q$ is in the corner equals $\alpha/(\alpha+\gamma)$,
and the asymptotic probability that the lowest point in the main row is not in the corner equals $c\beta/(\alpha+c\beta)$.
These events are asymptotically independent, so
\[
\prob{\text{$Q$ can't dance}} \;\sim\;
P_1 \eq \frac{\alpha}{\alpha+\gamma} \times \frac{c\beta}{\alpha+c\beta}
\eq
\frac{c\alpha\lambda}{\alpha+\gamma}
\qquad
\text{and}
\qquad
P_2 \eq 1-P_1 .
\]
Hence,
\[
\kappa(L_1)
\eq  P_1 + \tfrac12 P_2
\eq  \frac12\! \left(1+\frac{c \alpha \lambda}{\alpha + \gamma}\right)
\eq  \frac{2 r}{3 r-c+q-1} .
\]

\subsubsection*{\normalsize Corner type \texorpdfstring{$L_4$}{L4}}

\begin{center}
{\setgcptsize{.125}
\setgcextra{\gchdots{2}{2}\gcvdots{0}{0}}\gctwo2{1,-1}{-1} \quad
\setgcextra{\gchdots{-2}{2}\gchdots{3}{2}\gcvdots{1}{0}}\gctwo3{1,1,-1}{0,-1}}
\end{center}

This corner type has two peaks.
It can't occur in an \textsf{X}-shaped class.
Given a constrained gridded permutation,
let $Q_1$ be the peak point at the left and $Q_2$ be the peak point at the top.
The same analysis as for $L_1$ gives
\[
\prob{\text{$Q_1$ can't dance}} \;\sim\; \frac{c\alpha\lambda}{\alpha+\gamma},
\qquad
\prob{\text{$Q_2$ can't dance}} \;\sim\; \frac{r\alpha\lambda}{\alpha+\beta}.
\]
The event that $Q_1$ can dance depends on the points adjacent to the row divider below the corner,
whereas the event that $Q_2$ can dance depends on the points adjacent to the column divider to the right of the corner.
Thus these events
are asymptotically independent.
So, by Observation~\ref{obsIndepDancing}, we have
\[
\kappa(L_4)
\eq \kappa(L_1) \, \kappa(\rot{L_1})
\eq \frac14 \! \left(1+\frac{r \alpha \lambda}{\alpha + \beta}\right) \left(1+\frac{c \alpha \lambda}{\alpha + \gamma}\right)
\eq \frac{4rc}{(3r-c+q-1)(3c-r+q-1)} .
\]

\subsubsection*{\normalsize Corner type \texorpdfstring{$T_5$}{T5}}

\begin{center}
{\setgcptsize{.125}
\setgcextra{\gchdots{-2}{2}\gchdots{3}{2}\gcvdots{1}{0}}\gctwo3{-1,1,-1}{0,1} \quad
\raisebox{3.67pt}{\setgcextra{\gchdots{-2}{2}\gchdots{3}{2}\gcvdots{1}{0}\gcvdots{1}{5}}\gcthree3{0,1}{-1,1,-1}{0,1}}}
\end{center}

This corner type also has two peaks.
It can't occur in an \textsf{L}-shaped class.
Given a constrained gridded permutation,
let $Q_1$ and $Q_2$ be the two peak points.
The same analysis as for $L_1$ gives
\[
\prob{\text{$Q_1$ can't dance}} \;\sim\;
\prob{\text{$Q_2$ can't dance}} \;\sim\; \frac{r\alpha\lambda}{\alpha+\beta}.
\]
The event that one of the peak points can dance depends on the points adjacent to the column divider to the left of the corner,
whereas the event that the other peak point can dance depends on the points adjacent to the column divider to the right of the corner.
Thus these events
are asymptotically independent.
So, by
Observation~\ref{obsIndepDancing}, we have
\[
\kappa(T_5)
\eq \kappa(\rot{L_1})^2
\eq  \frac14\! \left(1+\frac{r \alpha \lambda}{\alpha + \beta}\right)^{\!2}
\eq  \frac{4c^2}{(3c-r+q-1)^2} .
\]

\subsection{Corners with diagonals}\label{sectDiagonalCorners}

\subsubsection*{\normalsize Corner type \texorpdfstring{$L_7$}{L7}}

\begin{center}
{\setgcptsize{.125}
\setgcextra{\gchdots{2}{2}\gcvdots{0}{0}}\gctwo2{1,1}{1} \quad
\setgcextra{\gchdots{-2}{2}\gchdots{3}{2}\gcvdots{1}{0}}\gctwo3{1,1,1}{0,1}}
\end{center}

This corner type has one diagonal.
It can't occur in an \textsf{X}-shaped class.
Let $\CR$ be the cell immediately to the right of the corner, and $\CB$ be the cell immediately below the corner.

\begin{figure}[ht]
  \centering
  \begin{tikzpicture}[scale=0.225]
  \plotgriddedperm{16}{12,1,3,13,2,16,4,5,6,7,8,11,14,9,10,15}{9,13}{2,6}
  \circpt74 \circpt85 \circpt96 \circpt{10}7 \circpt{11}8
  \node at (10.6,15.4) {\darkgrey\footnotesize$\CR$};
  \node at (1.6,5.4) {\darkgrey\footnotesize$\CB$};
  \node at (15,8) {\footnotesize$Q_1$};
  \node at (7,15) {\footnotesize$Q_2$};
  \end{tikzpicture}
  \caption{A \!\gcthree{3}{1,1,1}{1}{1}\!-gridded permutation; the five circled points, below $Q_1$ and to the right of $Q_2$, can dance; $k_1=2$ and $k_2=3$}\label{figL7}
\end{figure}

Given a constrained gridded permutation,
let $Q_1$ be the lowest point in the main row that is not in $\CR$.
Note that $Q_1$ may be in the corner cell.
Let $k_1$ be the number of points in $\CR$ that lie below $Q_1$.
Similarly, let $Q_2$ be the rightmost point in the main column that is not in $\CB$ ($Q_2$~may also be in the corner cell),
and let $k_2$ be the number of points in $\CB$ that lie to the right of $Q_2$.
These $k_1+k_2$ points can dance diagonally, giving $k_1+k_2+1$ distinct griddings of the underlying permutation.
Note that either or both of $k_1$ and $k_2$ may be zero.
Points above $Q_1$ and points to the left of $Q_2$ can't dance.
Points $Q_1$ and $Q_2$ are the controllers, because they control which points can dance.
See Figure~\ref{figL7} for an illustration.

Now, by Observation~\ref{obsProbs},
for each $i\geqs0$, we have
\[
\prob{k_1=i} \;\sim\; \lambda^i(1-\lambda) \text{\qquad and also \qquad} \prob{k_2=i} \;\sim\; \lambda^i(1-\lambda) .
\]
The value of $k_1$ is asymptotically independent of the value of~$k_2$.
Thus, for each $\ell\geqs1$, the asymptotic probability of having exactly $\ell$ griddings is
\[
  \prob{k_1+k_2+1=\ell} \eq \sum_{k_1=0}^{\ell-1} \prob{k_1=i}\prob{k_2=\ell-1-i} \;\sim\; \ell\,\lambda^{\ell-1}(1-\lambda)^2 .
\]
Hence,
\[
\kappa(L_7) \eq \sum_{\ell=1}^\infty \lambda^{\ell-1}(1-\lambda)^2
\eq 1-\lambda
\eq 1- \frac{r+c+1-q}{2 r c}.
\]

\subsubsection*{\normalsize Corner type \texorpdfstring{$T_2$}{T2}}

\begin{center}
{\setgcptsize{.125}
\setgcextra{\gchdots{-2}{2}\gchdots{3}{2}\gcvdots{1}{0}}\gctwo3{-1,-1,1}{0,-1} \quad
\raisebox{3.67pt}{\setgcextra{\gchdots{-2}{2}\gchdots{3}{2}\gcvdots{1}{0}\gcvdots{1}{5}}\gcthree3{0,-1}{-1,-1,1}{0,-1}}}
\end{center}

This corner type has one peak and one diagonal.
It can't occur in an \textsf{L}-shaped class.

Dancing at the peak depends on the points adjacent to the column divider to the right of the corner,
whereas dancing on the diagonal
depends on the points adjacent to the column divider to the left of the corner and
points adjacent to the row divider below the corner.
So these
are asymptotically independent.
Thus, from the analysis for $L_1$ and $L_7$, and by Observation~\ref{obsIndepDancing}, we have
\[
\kappa(T_2)
\eq \kappa(\rot{L_1}) \, \kappa(L_7)
\eq \frac{1}{2}\! \left(1+\frac{r \alpha \lambda}{\alpha + \beta}\right) (1-\lambda)
. 
\]

\subsubsection*{\normalsize Corner type \texorpdfstring{$X_0$}{X0}}

\begin{center}
{\setgcptsize{.125}\setgcextra{\gchdots{-2}{2}\gchdots{3}{2}\gcvdots{1}{0}\gcvdots{1}{5}}\gcthree3{0,-1}{-1,-1,-1}{0,-1}}
\end{center}

This corner type has two diagonals.
It can't occur in an \textsf{L}-shaped or \textsf{T}-shaped class.

Dancing on one of the diagonals depends on the points adjacent to the column divider
to the left of the corner and points adjacent to the row divider below the corner,
whereas
dancing on the other diagonal depends on the points adjacent to the column divider
to the right of the corner and points adjacent to the row divider above the corner.
So these
are asymptotically independent.
Thus, from the analysis for $L_7$, and by Observation~\ref{obsIndepDancing}, we have
\[
\kappa(X_0)
\eq \kappa(L_7)^2
\eq (1-\lambda)^2
. 
\]

\subsubsection*{\normalsize Corner type \texorpdfstring{$X_7$}{X7}}

\begin{center}
{\setgcptsize{.125}\setgcextra{\gchdots{-2}{2}\gchdots{3}{2}\gcvdots{1}{0}\gcvdots{1}{5}}\gcthree3{0,-1}{-1,1,1}{0,1}}
\end{center}

This corner type has two peaks and one diagonal.
It can't occur in an \textsf{L}-shaped or \textsf{T}-shaped class.

Dancing at the two peaks depends on the points adjacent to the column divider
to the left of the corner and points adjacent to the row divider above the corner,
whereas
dancing on the diagonal depends on the points adjacent to the column divider
to the right of the corner and points adjacent to the row divider below the corner.
So these
are asymptotically independent.
Thus, from the analysis for $L_4$ and $L_7$, and by Observation~\ref{obsIndepDancing}, we have
\[
\kappa(X_7)
\eq \kappa(L_4)\,\kappa(L_7)
\eq \frac{1}{4}\! \left(1+\frac{r \alpha \lambda}{\alpha + \beta}\right) \left(1+\frac{c \alpha \lambda}{\alpha + \gamma}\right) (1-\lambda)
. 
\]

\subsection{Corners with tees}\label{sectTeeCorners}

\subsubsection*{\normalsize Corner type \texorpdfstring{$L_3$}{L3}}

\begin{center}
{\setgcptsize{.125}
\setgcextra{\gchdots{2}{2}\gcvdots{0}{0}}\gctwo2{-1,1}{1} \quad
\setgcextra{\gchdots{-2}{2}\gchdots{3}{2}\gcvdots{1}{0}}\gctwo3{-1,-1,1}{0,1} \quad
\raisebox{3.67pt}{\setgcextra{\gchdots{-2}{2}\gchdots{3}{2}\gcvdots{1}{0}\gcvdots{1}{5}}\gcthree3{0,-1}{-1,-1,1}{0,1}}}
\end{center}

This corner type has one tee.
It is the only one of the eleven corner types that can occur in \textsf{L}-shaped, \textsf{T}-shaped \emph{and} \textsf{X}-shaped classes.

Analysis of this corner type is considerably more complex than for the corners we've considered so far.
We analyse the main row by reading its points from the bottom.
Similarly, we analyse the main column by reading its points from the right.
Let $C$ denote the corner cell,
let $\CR$ be the cell immediately to the right of~$C$,
and let $\CB$ be the cell immediately below~$C$.

\medskip
\noindent
\textbf{Sequences of points.}
We represent the sequence of points in the main row, when read from the bottom, by words over the alphabet $\{\xxx,\yyy,\zzz\}$.
The letter $\xxx$ represents a point in the corner cell~$C$,
the letter $\yyy$ represents a point in~$\CR$,
and the letter $\zzz$ represents a point in some other cell of the main row.
We use subscripts to identify specific points.
Thus $\xxx_Q$ represents the occurrence of point~$Q$ in the corner cell.
For example, the sequence of points in the main row of the gridded permutation at the left of Figure~\ref{figL3Case1} 
is represented by the word
$
\xxx_Q\,\yyy\,\xxx_R\,\yyy\,\yyy\,\zzz_S\,\xxx\,\zzz\,\yyy
$.

Simple \emph{regular expressions} are used to denote sets specifying the possible ordering of the initial points.
In these regular expressions, $a^\STAR$ denotes a sequence of zero or more copies of letter~$a$.
For example, $\zzz^\STAR \xxx_Q$ consists of arrangements of points in the main row in which $Q$ is the lowest point in the corner cell ($\xxx_Q$~being the first occurrence of~$\xxx$) and where $Q$ is below any points in~$\CR$, the only points below $Q$ occurring elsewhere (represented by~$\zzz^\STAR$).
Note that any ordering of points is permitted after the initial points specified by the regular expression.

By Observation~\ref{obsProbs}, associated with $\xxx$, $\yyy$ and $\zzz$ we have the following asymptotic probabilities:
\newcommand{\ppx}{1-c\lambda}
\newcommand{\ppy}{\lambda}
\newcommand{\ppz}{(c-1)\lambda}
\[
\px \eq \ppx,  \qquad
\py \eq \ppy,  \qquad
\pz \eq \ppz.
\]
By symmetry, for the location of points in the main column, when read from the right, the same approach yields the following asymptotic probabilities:
\[
\qx \eq 1-r\lambda,  \qquad
\qy \eq \lambda,  \qquad
\qz \eq (r-1)\lambda.
\]

The asymptotic probability that the points in the main row are in (the set denoted by) some regular expression is given by the product formed by replacing each $\xxx$ by $\px$, each $\yyy$ by $\py$ and each $\zzz$ by $\pz$, and replacing each $\xxx^\STAR$, $\yyy^\STAR$ and $\zzz^\STAR$ by $1/(1-\px)$, $1/(1-\py)$ and $1/(1-\pz)$, respectively.
For example, the asymptotic probability that the arrangement of points in the main row is in $\zzz^\STAR \xxx_Q$ equals $\px/(1-\pz)$.

\medskip
\noindent
\textbf{The number of griddings.}
Recall, from the introduction to tee dancing on page~\pageref{defDiagDancing}, that a monotone sequence of points can dance.
The points not at one of the ends of the sequence can dance between $\CR$ and $\CB$ through~$C$.
However, the first and last of these points may only be able to dance between $\CR$ and $C$ or between $\CB$ and~$C$, but not from $\CR$ to~$\CB$.

Each point that can dance between $\CR$ and $\CB$ contributes two to the number of 
griddings.
However, a point that can only dance between $C$ and either $\CR$ or $\CB$ contributes just one additional gridding.
In our analysis below, we determine the contribution to the number of griddings from points in the main row, which we denote $m_1$,
and also the contribution to the number of griddings from points in the main column, which we denote $m_2$.
Including the original gridding, the total number of distinct griddings is then $m_1+m_2+1$.
For example, in the gridded permutation towards the centre of Figure~\ref{figM1M2}, we have $m_1=3$ (from points in the main row dancing down and to the left) and $m_2=2$ (from points in the main column dancing to the right and up), yielding a total of six distinct griddings.

\begin{figure}[ht]
  \centering
  \begin{tikzpicture}[scale=0.225]
  \plotgriddedperm{7}{6,2,3,1,4,5,7}{6}{1,5}
  \circpt{3}{3} \circpt{5}{4} \circpt{6}{5}
  \end{tikzpicture}
  \qquad
  \begin{tikzpicture}[scale=0.225]
  \plotgriddedperm{7}{6,2,3,1,4,5,7}{6}{1,4}
  \circpt{3}{3} \circpt{5}{4} \circpt{6}{5}
  \end{tikzpicture}
  \qquad
  \begin{tikzpicture}[scale=0.225]
  \plotgriddedperm{7}{6,2,3,1,4,5,7}{5}{1,4}
  \circpt{3}{3} \circpt{5}{4} \circpt{6}{5}
  \end{tikzpicture}
  \quad\raisebox{20pt}{$\Longleftarrow$}\quad
  \begin{tikzpicture}[scale=0.225]
  \plotgriddedperm{7}{6,2,3,1,4,5,7}{5}{1,3}
  \circpt{3}{3} \circpt{5}{4} \circpt{6}{5}
  \end{tikzpicture}
  \quad\raisebox{20pt}{$\Longrightarrow$}\quad
  \begin{tikzpicture}[scale=0.225]
  \plotgriddedperm{7}{6,2,3,1,4,5,7}{4}{1,3}
  \circpt{3}{3} \circpt{5}{4} \circpt{6}{5}
  \end{tikzpicture}
  \qquad
  \begin{tikzpicture}[scale=0.225]
  \plotgriddedperm{7}{6,2,3,1,4,5,7}{4}{1,2}
  \circpt{3}{3} \circpt{5}{4} \circpt{6}{5}
  \end{tikzpicture}
  \caption{The six griddings of 6231457 in \!\gcthree2{-1,1}{1}{1}\!; the circled points can dance}\label{figM1M2}
\end{figure}

\medskip
\noindent
\textbf{Labelled points.}
Given a constrained gridded permutation,
let $Q$ be the lowest point in the corner cell~$C$.
Our analysis depends on whether or not $Q$ is the peak point of each of the two peaks that form the tee.

If $c>1$ (so we have more than two columns), then let $S$ be the lowest point in the main row that is in a cell other than $C$ or~$\CR$.
Similarly, if $r>1$, then let $T$ be the rightmost point in the main column that is in a cell other than $C$ or~$\CB$.

We consider three cases.
In the first, $Q$ is the peak point of both of the peaks that form the tee.
In the second, $Q$ is not the peak point of either of the peaks that form the tee.
Finally, in the third, $Q$ is the peak point of just one of the two peaks. 

\subsubsection*{\normalsize Case~1: $Q$ is the peak point of both of the peaks that form the tee}

In Case~1, point $Q$ is lower than every point in $\CR$ and to the right of every point in~$\CB$.
If $S$ is above~$Q$, then $Q$ is adjacent to the row divider below $C$ and can dance vertically into~$\CB$, contributing one to the number of griddings.
Similarly, if $T$ is to the left of~$Q$, then $Q$ is adjacent to the column divider to the right of $C$ and can dance horizontally into~$\CR$, again contributing one to the number of griddings.

Let $R$ be the second lowest point in the corner cell~$C$.
That is, $R$ is the point immediately above $Q$ in the corner cell.
Points $R$, $S$  and $T$ are the controllers, because they control which points can dance.

Let $k_1$ be the number of points in $\CR$ below both $R$ and $S$ (or if $c=1$, the number below $R$).
These $k_1$ points are above~$Q$, and thus can all dance through the tee to $\CB$, each contributing 2 to the number of griddings.
Similarly, let $k_2$ be the number of points in $\CB$ to the right of both $R$ and $T$ (or if $r=1$, the number to the right of $R$).
These $k_2$ points are to the left of~$Q$, and can also all dance through to $\CR$, again each contributing 2 to the number of griddings.
Note that each of $k_1$ and $k_2$ may be zero.
See Figure~\ref{figL3Case1} for two gridded permutations satisfying the conditions of Case~1.

\begin{figure}[ht]
  \centering
  \begin{tikzpicture}[scale=0.225]
  \plotgriddedperm{16}{4,14,1,5,10,2,3,6,7,8,9,11,12,16,13,15}{10,14}{3,7}
  \node at (1.15,15.5) {\darkgrey\footnotesize$C$};
  \node at (11.6,15.4) {\darkgrey\footnotesize$\CR$};
  \node at (1.6,6.4) {\darkgrey\footnotesize$\CB$};
  \node at ( 9,9) {\footnotesize$Q$};
  \node at ( 6,11) {\footnotesize$R$};
  \node at (15.6,12) {\footnotesize$S$};
  \node at ( 8, 2.3) {\footnotesize$T$};
  \circpt{8}{6} \circpt{9}{7} \circpt{10}{8} \circpt{11}{9}
  \node at (8.5,-.8) {\footnotesize Main row: \textbf{1a} ($k_1=1$)};
  \node at (8.5,-2.6) {\footnotesize Main column: \textbf{1a} ($k_2=2$)};
  \node at (8.5,-4.4) {\footnotesize 9 griddings};
  \end{tikzpicture}
  $\qquad\qquad\quad$
  \begin{tikzpicture}[scale=0.225]
  \plotgriddedperm{16}{3,1,15,4,14,5,6,7,2,8,9,12,13,16,10,11}{9,14}{2,6}
  \node at (1.15,15.5) {\darkgrey\footnotesize$C$};
  \node at (10.6,15.4) {\darkgrey\footnotesize$\CR$};
  \node at (1.6,5.4) {\darkgrey\footnotesize$\CB$};
  \node at ( 7, 8) {\footnotesize$Q$};
  \node at ( 6,15) {\footnotesize$R$};
  \node at (15.6,9) {\footnotesize$S$};
  \node at ( 8, 1.3) {\footnotesize$T$};
  \node at (11.2,13) {\footnotesize$U$};
  \circpt{8}{7} \circpt{10}{8} \circpt{11}{9} \circpt{12}{12}
  \node at (8.5,-.8) {\footnotesize Main row: \textbf{1b} ($k_1=2$)};
  \node at (8.5,-2.6) {\footnotesize Main column: \textbf{1c} ($k_2=0$)};
  \node at (8.5,-4.4) {\footnotesize 7 griddings};
  \end{tikzpicture}
  \caption{$L_3$ Case~1: \!\gcthree{3}{-1,1,1}{1}{1}\!-gridded permutations;
  the circled points can dance}\label{figL3Case1}
\end{figure}

In Case~1, the possible ordering of the points in the main row, reading from the bottom, is given by the regular expression $\zzz^\STAR \xxx_Q$, the only points  below $Q$ (if any) being in cells other than $C$ and~$\CR$.
So, the asymptotic probability that the points in the main row of a gridded permutation satisfy the conditions of Case~1 is given by
    \[
    p_{\textbf{1}} \defeq
    \prob{\text{Case~1}} \eq
    \frac{\px}{1-\pz}
    \eq \frac{\ppx}{1-\ppz}
    .
    \]
An analogous result applies for the main column.
For both the main row and the main column, we have three subcases.
We analyse the main row, reading its points from the bottom.
The analysis of the main column is analogous (reading its points from the right).
The three subcases are as follows:

\begin{itemize}
    \item[\textbf{1a.}] 
    $S$,~if it exists, is above~$Q$, which is the lowest point in the main row.
    $S$~(if it exists) may be above or below~$R$.
    There is no point in $\CR$ that is above $S$ and below~$R$.
    See the left of Figure~\ref{figL3Case1} for examples in both the main row and the main column.

    $Q$~can dance vertically into~$\CB$ and $k_1$ points can dance from $\CR$ to $\CB$ via $C$. Thus $m_1=2k_1+1$.

    The possible ordering of the points in the main row, reading from the bottom, is given by the regular expression $\xxx_Q \yyy^\STAR \zzz^\STAR \xxx_R$.
    So, for each $i\geqs0$, we have
    \[
    p_{\textbf{1a}}(i) \defeq
    \prob{\text{Case \textbf{1a} and $k_1=i$}}
    \eq \frac{\px^2 \,\py^i}{1-\pz}
    \eq \frac{(\ppx)^2 \,\ppy^i}{1-\ppz}
    .
    \]

    \item[\textbf{1b.}] 
    $S$~is above $Q$ and below~$R$, and at least one point in $\CR$ is above $S$ and below~$R$.
    Again, $Q$~is the lowest point in the main row.
    Let $U$ be the lowest of the points in $\CR$ above $S$ and below~$R$.
    Note that $U$ is not one of the $k_1$ points in $\CR$ below both $R$ and~$S$.
    Other points not in $C$ or $\CR$ may lie below~$U$.
    See the right of Figure~\ref{figL3Case1} for an example.

    $Q$~can dance vertically into~$\CB$, contributing one to the number of griddings.
    And $U$ can dance horizontally into the corner cell, but not from the corner into~$\CB$, so this contributes another gridding.
    Thus $m_1=2k_1+2$.

    The possible ordering of the points in the main row, reading from the bottom, is given by the regular expression $\xxx_Q \yyy^\STAR \zzz_S \zzz^\STAR \yyy_U$.
    So, for each $i\geqs0$, we have
    \[
    p_{\textbf{1b}}(i) \defeq
    \prob{\text{Case \textbf{1b} and $k_1=i$}}
    \eq \frac{\px \,\py^{i+1} \,\pz}{1-\pz}
    \eq \frac{(\ppx) \,\ppy^{i+1} \,\ppz}{1-\ppz}
    .
    \]

    \item[\textbf{1c.}]
    $S$~is below~$Q$.
    So $Q$ can't dance vertically into~$\CB$, and $k_1=0$, since every point in $\CR$ is above $Q$ and hence also above~$S$.
    Thus $m_1=0$.
    See the right of Figure~\ref{figL3Case1} for an example of this in the main \emph{column},
    in which the controller $T$ is to the right of~$Q$, so $Q$ can't dance horizontally into~$\CR$, and $m_2=k_2=0$.

    The possible ordering of the points in the main row, reading from the bottom, is given by the regular expression $\zzz_S \zzz^\STAR \xxx_Q$.
    Thus,
    \[
    p_{\textbf{1c}} \defeq
    \prob{\text{Case \textbf{1c} and $k_1=0$}}
    \eq \frac{\px \,\pz}{1-\pz}
    \eq \frac{(\ppx) \,\ppz}{1-\ppz}
    .
    \]

\end{itemize}

Table~\ref{tblL3Case1} gives the total number of griddings for each possibility in Case~1.

\begin{table}[ht]
  \centering
  \small
  \renewcommand{\arraystretch}{1.25}
  \begin{tabular}{ll|c|c|c|}
                            &                       & Main row \textbf{1a}  & Main row \textbf{1b}  & Main row \textbf{1c}  \\
                            &                       & \darkgrey$m_1=2k_1+1$ & \darkgrey$m_1=2k_1+2$ & \darkgrey$m_1=0$      \\\hline
    Main column \textbf{1a} & \darkgrey$m_2=2k_2+1$ & $2k_1+2k_2+3$         & $2k_1+2k_2+4$         & $2k_2+2$              \\\hline
    Main column \textbf{1b} & \darkgrey$m_2=2k_2+2$ & $2k_1+2k_2+4$         & $2k_1+2k_2+5$         & $2k_2+3$              \\\hline
    Main column \textbf{1c} & \darkgrey$m_2=0$      & $2k_1+2$              & $2k_1+3$              & $1$                   \\\hline
  \end{tabular}
  \caption{The number of griddings for each combination of subcases in Case~1}\label{tblL3Case1}
\end{table}

We now the calculate, for each $\ell\geqs1$, the probability $P_{\textbf{1}}(\ell) := \prob{\text{Case \textbf{1} and $\ell$ griddings}}$.

Let $q_{\textbf{1a}}(i)$, $q_{\textbf{1b}}(i)$ and $q_{\textbf{1c}}$ be the subcase probabilities for the main column,
formed from $p_{\textbf{1a}}(i)$, $p_{\textbf{1b}}(i)$ and $p_{\textbf{1c}}$
by replacing $\px$, $\py$, $\pz$ with
$\qx$, $\qy$, $\qz$,
respectively.

Then, from Table~\ref{tblL3Case1}, we have
\begin{align*}
P_{\textbf{1}}(1) & \eq p_{\textbf{1c}} \, q_{\textbf{1c}} , \\[3pt]
P_{\textbf{1}}(2) & \eq p_{\textbf{1a}}(0) \, q_{\textbf{1c}} \:+\: p_{\textbf{1c}} \, q_{\textbf{1a}}(0) , \\[3pt]
P_{\textbf{1}}(3) & \eq p_{\textbf{1a}}(0) \, q_{\textbf{1a}}(0) \:+\: p_{\textbf{1b}}(0) \, q_{\textbf{1c}} \:+\: p_{\textbf{1c}} \, q_{\textbf{1b}}(0) , \\[3pt]
P_{\textbf{1}}(\ell) & \eq \phantom{\:+\:} p_{\textbf{1a}}\big(\tfrac{\ell-2}2\big) \, q_{\textbf{1c}} \:+\: p_{\textbf{1c}} \, q_{\textbf{1a}}\big(\tfrac{\ell-2}2\big) \\
                       & \phantom{\eq} \:+\: \sum_{i=0}^{(\ell-4)/2}
                       \left( p_{\textbf{1a}}(i) \, q_{\textbf{1b}}\big(\tfrac{\ell-4}2-i\big)
                       \:+\: p_{\textbf{1b}}\big(\tfrac{\ell-4}2-i\big) \, q_{\textbf{1a}}(i) \right)
                       , \quad \text{$\ell\geqs4$, even} , \\[3pt]
P_{\textbf{1}}(\ell) & \eq \phantom{\:+\:} p_{\textbf{1b}}\big(\tfrac{\ell-3}2\big) \, q_{\textbf{1c}} + p_{\textbf{1c}} \, q_{\textbf{1b}}\big(\tfrac{\ell-3}2\big) \\
                       & \phantom{\eq} \:+\: \sum_{i=0}^{(\ell-3)/2} p_{\textbf{1a}}(i) \, q_{\textbf{1a}}(\tfrac{\ell-3}2-i)
                                       \:+\: \sum_{i=0}^{(\ell-5)/2} p_{\textbf{1b}}(i) \, q_{\textbf{1b}}(\tfrac{\ell-5}2-i)
                       , \quad \text{$\ell\geqs5$, odd} .
\end{align*}

After simplification, facilitated by using a computer algebra system, this yields
\[
P_{\textbf{1}}(\ell) \eq
\begin{cases}
\dfrac{\lambda^{(\ell-3)/2}\, \px\, \qx\, \big((\ell-1)\,\px\,\qx + (\ell+1)\,\lambda\,\pz\,\qz \big) }{2\, (1-\pz)\, (1-\qz)} , & \text{$\ell\geqs1$, odd} , \\[9pt]
\dfrac{\ell\, \lambda^{(\ell-2)/2}\, \px\, \qx\, (\px\, \qz + \pz\, \qx)}{2\, (1-\pz)\, (1-\qz)} , & \text{$\ell\geqs2$, even} .
\end{cases}
\]

\subsubsection*{\normalsize Case~2: $Q$ is not the peak point of either of the peaks that form the tee}

In Case~2, there is a point in $\CR$ below $Q$ and a point in $\CB$ to the right of $Q$.
So $Q$ can't dance into either $\CB$ or $\CR$.
Points $Q$, $S$  and $T$ are the controllers, controlling which points can dance.

Let $k_1$ be the number of points in $\CR$ below both $Q$ and $S$ (or if $c=1$, the number below $Q$).
These $k_1$ points can all dance through the tee to $\CB$, each contributing 2 to the number of griddings.
Similarly, let $k_2$ be the number of points in $\CB$ to the right of both $Q$ and $T$ (or if $r=1$, the number to the right of $Q$).
These $k_2$ points can also all dance through to $\CR$, again each contributing 2 to the number of griddings.
See Figure~\ref{figL3Case2} for a gridded permutation satisfying the conditions of Case~2.
Again, $k_1$ and $k_2$ may be zero.

\begin{figure}[ht]
  \centering
  \begin{tikzpicture}[scale=0.225]
  \plotgriddedperm{16}{3,1,14,4,11,5,2,6,7,8,9,12,13,16,10,15}{9,14}{2,7}
  \node at (1.15,15.5) {\darkgrey\footnotesize$C$};
  \node at (10.6,15.4) {\darkgrey\footnotesize$\CR$};
  \node at (1.6,6.4) {\darkgrey\footnotesize$\CB$};
  \node at ( 6,12) {\footnotesize$Q$};
  \node at (15.6, 9) {\footnotesize$S$};
  \node at ( 6, 1.3) {\footnotesize$T$};
  \circpt{6}{5} \circpt{8}{6} \circpt{9}{7} \circpt{10}{8} \circpt{11}{9}
  \node at (8.5,-.8) {\footnotesize Main row: \textbf{2a} ($k_1=2$)};
  \node at (8.5,-2.6) {\footnotesize Main column: \textbf{2b} ($k_2=2$)};
  \node at (8.5,-4.4) {\footnotesize 10 griddings};
  \end{tikzpicture}
  \caption{$L_3$ Case~2: a \!\gcthree{3}{-1,1,1}{1}{1}\!-gridded permutation;
  the circled points can dance}\label{figL3Case2}
\end{figure}

For both the main row and the main column, we have two subcases.
We analyse the main row, the analysis of the main column being analogous.

\begin{itemize}
    \item[\textbf{2a.}] 
    There is no point in $\CR$ above $S$ and below~$Q$.
    $S$~may be above or below $Q$, or $S$ may not exist.
    Since there is a point in $\CR$ below $Q$ (by the definition of Case~2), in this subcase (only) we know that $k_1>0$.

    The $k_1$ points in $\CR$ below both $Q$ and $S$ (if it exists) can all dance through to~$\CB$.
    Thus $m_1=2k_1$.

    The possible ordering of the points in the main row, reading from the bottom, is given by the regular expression $\yyy\yyy^\STAR \zzz^\STAR \xxx_Q$.
    So, for each $i\geqs1$, we have
    \[
    p_{\textbf{2a}}(i) \defeq
    \prob{\text{Case \textbf{2a} and $k_1=i$}}
    \eq \frac{\px \,\py^i}{1-\pz}
    \eq \frac{(\ppx) \,\ppy^i}{1-\ppz}
    .
    \]

    \item[\textbf{2b.}] 
      $S$ is below $Q$ with at least one point in $\CR$ above $S$ and below $Q$.
      Let $U$ be the lowest of these points.
      Note that $U$ is not one of the $k_1$ points in $\CR$ that are below both $Q$ and~$S$.
      Other points not in $C$ or $\CR$ may lie below $U$.

      $U$ can dance horizontally into the corner cell, but not from the corner into~$\CB$.
      Thus $m_1=2k_1+1$.

    The possible ordering of the points in the main row, reading from the bottom, is given by the regular expression $\yyy^\STAR \zzz_S \zzz^\STAR \yyy_U$.
    So, for each $i\geqs0$, we have
    \[
    p_{\textbf{2b}}(i) \defeq
    \prob{\text{Case \textbf{2b} and $k_1=i$}}
    \eq \frac{\py^{i+1} \,\pz}{1-\pz}
    \eq \frac{\ppy^{i+1} \,\ppz}{1-\ppz}
    .
    \]
\end{itemize}

Table~\ref{tblL3Case2} gives the total number of griddings for each possibility in Case~2.

\begin{table}[ht]
  \centering
  \small
  \renewcommand{\arraystretch}{1.25}
  \begin{tabular}{ll|c|c|c|}
                            &                       & Main row \textbf{2a}  & Main row \textbf{2b}  \\
                            &                       & \darkgrey$m_1=2k_1>0$   & \darkgrey$m_1=2k_1+1$ \\\hline
    Main column \textbf{2a} & \darkgrey$m_2=2k_2>0$   & $2k_1+2k_2+1$         & $2k_1+2k_2+2$         \\\hline
    Main column \textbf{2b} & \darkgrey$m_2=2k_2+1$ & $2k_1+2k_2+2$         & $2k_1+2k_2+3$         \\\hline
  \end{tabular}
  \caption{The number of griddings for each combination of subcases in Case~2}\label{tblL3Case2}
\end{table}

We now calculate, for each $\ell\geqs1$, the probability $P_{\textbf{2}}(\ell) = \prob{\text{Case \textbf{2} and $\ell$ griddings}}$.

Let $q_{\textbf{2a}}(i)$ and $q_{\textbf{2b}}(i)$ be the subcase probabilities for the main column,
formed from $p_{\textbf{2a}}(i)$ and $p_{\textbf{2b}}(i)$
by replacing $\px$, $\py$, $\pz$ with
$\qx$, $\qy$, $\qz$,
respectively.

Then, from Table~\ref{tblL3Case2},
and recalling that in subcase \textbf{2a} there is always at least one point that can dance,
we have
\begin{align*}
P_{\textbf{2}}(3) & \eq p_{\textbf{2b}}(0) \, q_{\textbf{2b}}(0) , \\[3pt]
P_{\textbf{2}}(\ell) & \eq \sum_{i=1}^{(\ell-2)/2}
                        \left( p_{\textbf{2a}}(i) \, q_{\textbf{2b}}\big(\tfrac{\ell-2}2-i\big)
                        \:+\: p_{\textbf{2b}}\big(\tfrac{\ell-2}2-i\big) \, q_{\textbf{2a}}(i) \right)
                        , \quad \text{$\ell\geqs4$, even} , \\[3pt]
P_{\textbf{2}}(\ell) & \sum_{i=1}^{(\ell-1)/2-1} p_{\textbf{2a}}(i) \, q_{\textbf{2a}}(\tfrac{\ell-1}2-i)
                       \:+\: \sum_{i=0}^{(\ell-3)/2} p_{\textbf{2b}}(i) \, q_{\textbf{2b}}(\tfrac{\ell-3}2-i)
                       , \quad \text{$\ell\geqs5$, odd} .
\end{align*}

After simplification, this yields
\[
P_{\textbf{2}}(\ell) \eq
\begin{cases}
0 , & \ell=1 , \\[4pt]
\dfrac{(\ell-2)\, \lambda^{\ell/2}\, (\px\, \qz + \pz\, \qx)}{2\, (1-\pz)\, (1-\qz)} , & \text{$\ell\geqs2$, even} , \\[9pt]
\dfrac{\lambda^{(\ell-1)/2}\, \big((\ell-3)\,\px\,\qx + (\ell-1)\,\lambda\,\pz\,\qz \big) }{2\, (1-\pz)\, (1-\qz)} , & \text{$\ell\geqs3$, odd} .
\end{cases}
\]

\subsubsection*{\normalsize Case~3: $Q$ is the peak point of just one of the two peaks that form the tee}

Case~3 combines Case~1 for the main row and Case~2 for the main column, or vice versa.

Suppose the main row satisfies Case~1 and the main column satisfies Case~2, as in Figure~\ref{figL3Case3}.
Then $Q$ can't dance into either $\CR$ or $\CB$ because of the points to its right in~$\CB$.
Neither can the points in $\CR$ dance, because $Q$ is below them. So $m_1=0$.
On the other hand, the Case~2 analysis of the main column is still valid, the $k_2$ points in $\CB$ to the right of $Q$ being able to dance through to~$\CR$.

The situation is analogous if the main row satisfies Case~2 and the main column satisfies Case~1.
Table~\ref{tblL3Case3} gives the total number of griddings for each possibility in Case~3.

\begin{figure}[ht]
  \centering
  \begin{tikzpicture}[scale=0.225]
  \plotgriddedperm{16}{3,1,14,4,5,2,8,6,7,9,10,11,13,16,12,15}{9,14}{2,7}
  \node at (1.15,15.5) {\darkgrey\footnotesize$C$};
  \node at (10.6,15.4) {\darkgrey\footnotesize$\CR$};
  \node at (1.6,6.4) {\darkgrey\footnotesize$\CB$};
  \node at ( 6,9) {\footnotesize$Q$};
  \node at (15.6, 11) {\footnotesize$S$};
  \node at ( 5, 1.3) {\footnotesize$T$};
  \circpt{8}{6} \circpt{9}{7}
  \node at (8.5,-.8) {\footnotesize Main row: \textbf{1} ($k_1=0$)};
  \node at (8.5,-2.6) {\footnotesize Main column: \textbf{2a} ($k_2=2$)};
  \node at (8.5,-4.4) {\footnotesize 5 griddings};
  \end{tikzpicture}
  \caption{$L_3$ Case~3: a \!\gcthree{3}{-1,1,1}{1}{1}\!-gridded permutation;
  the circled points can dance}\label{figL3Case3}
\end{figure}

\begin{table}[ht]
  \centering
  \small
  \renewcommand{\arraystretch}{1.25}
  \begin{tabular}{ll|c|c|c|c|}
                            &                       & Main row \textbf{1} & Main row \textbf{2a}  & Main row \textbf{2b}  \\
                            &                       & \darkgrey$m_1=0$    & \darkgrey$m_1=2k_1$   & \darkgrey$m_1=2k_1+1$ \\\hline
    Main column \textbf{1}  & \darkgrey$m_2=0$      &                     & $2k_1+1$              & $2k_1+2$              \\\hline
    Main column \textbf{2a} & \darkgrey$m_2=2k_2$   & $2k_2+1$            &                       &                       \\\hline
    Main column \textbf{2b} & \darkgrey$m_2=2k_2+1$ & $2k_2+2$            &                       &                       \\\hline
  \end{tabular}
  \caption{The number of griddings for each combination of subcases in Case~3}\label{tblL3Case3}
\end{table}

Thus, for each $\ell\geqs1$, the probability $P_{\textbf{3}}(\ell) = \prob{\text{Case \textbf{3} and $\ell$ griddings}}$ is given by
\[
P_{\textbf{3}}(\ell) \eq
\begin{cases}
0 , & \ell=1 , \\[4pt]
p_{\textbf{2b}}(\tfrac{\ell-2}2) \, q_{\textbf{1}} \:+\: p_{\textbf{1}} \, q_{\textbf{2b}}(\tfrac{\ell-2}2) \;=\;
                        \dfrac{\lambda^{\ell/2}\, (\px\, \qz + \pz\, \qx)}{(1-\pz)\, (1-\qz)}
                        , & \text{$\ell\geqs2$, even} , \\[9pt]
p_{\textbf{2a}}(\tfrac{\ell-1}2) \, q_{\textbf{1}} \:+\: p_{\textbf{1}} \, q_{\textbf{2a}}(\tfrac{\ell-1}2) \:\;=\;
                        \dfrac{2\, \lambda^{(\ell-1)/2}\, \px\, \qx}{(1-\pz)\, (1-\qz)}
                        , & \text{$\ell\geqs3$, odd} ,
\end{cases}
\]
where $q_{\textbf{1}}$ is the probability that the main column satisfies Case~1,
formed from $p_{\textbf{1}}$
by replacing $\px$ and $\pz$ with
$\qx$ and $\qz$,
respectively.

We can now combine the three cases.
The asymptotic probability that the underlying permutation of a gridded permutation has exactly $\ell$ griddings is
\[
P_\ell \eq P_{\textbf{1}}(\ell) + P_{\textbf{2}}(\ell) + P_{\textbf{3}}(\ell) .
\]
After considerable simplification, facilitated by using a computer algebra system, we have
\[
P_\ell \eq
\begin{cases}
\dfrac{\px\, \pz\, \qx\, \qz}{(1-\pz)\, (1-\qz)} , & \ell=1 ,\\[10.5pt]
\dfrac{\ell\, \lambda^{(\ell-2)/2} \, (\lambda + \px\, \qx) \, (\px\, \qz + \pz\, \qx)}{2\, (1-\pz)\, (1-\qz)} , & \text{$\ell\geqs2$, even} , \\[9pt]
\dfrac{\ell\, \lambda^{(\ell-3)/2} \, (\lambda + \px\, \qx) \, (\px\, \qx + \lambda\, \pz\, \qz)}{2\, (1-\pz)\, (1-\qz)} , & \text{$\ell\geqs3$, odd} .
\end{cases}
\]

Finally, after further simplification, this gives us the correction factor for the $L_3$ corner type:
\[
\kappa(L_3) \eq \sum_{\ell\geqs1}P_\ell/\ell
\eq
\frac{\lambda\,(1-\lambda)}{(1-\pz)\, (1-\qz)}
\eq
\frac{\lambda\,(1-\lambda)}{(1-(c-1)\lambda)\, (1-(r-1)\lambda)} .
\]

\subsubsection*{\normalsize Corner type \texorpdfstring{$T_4$}{T4}}

\begin{center}
{\setgcptsize{.125}
\setgcextra{\gchdots{-2}{2}\gchdots{3}{2}\gcvdots{1}{0}}\gctwo3{-1,1,-1}{0,-1} \quad
\raisebox{3.67pt}{\setgcextra{\gchdots{-2}{2}\gchdots{3}{2}\gcvdots{1}{0}\gcvdots{1}{5}}\gcthree3{0,1}{-1,1,-1}{0,-1}}}
\end{center}

This corner type has one peak and one tee.
It can't occur in an \textsf{L}-shaped class.

Dancing at the peak depends on the points adjacent to the column divider
to the right of the corner,
whereas
tee dancing depends on the points adjacent to the column divider
to the left of the corner and points adjacent to the row divider below the corner.
So these
are asymptotically independent.
Thus, from the analysis for $L_1$ and $L_3$, and by Observation~\ref{obsIndepDancing}, we have
\[
\kappa(T_4)
\eq \kappa(\rot{L_1}) \, \kappa(L_3)
\eq \frac12 \! \left(1+\frac{r \alpha \lambda}{\alpha + \beta}\right) \frac{\lambda\,(1-\lambda)}{(1-(c-1)\lambda)\, (1-(r-1)\lambda)} .
\]

\subsubsection*{\normalsize Corner type \texorpdfstring{$X_4$}{X4}}

\begin{center}
{\setgcptsize{.125}\setgcextra{\gchdots{-2}{2}\gchdots{3}{2}\gcvdots{1}{0}\gcvdots{1}{5}}\gcthree3{0,-1}{-1,1,-1}{0,-1}}
\end{center}

This corner type has two tees.
It can't occur in an \textsf{L}-shaped or \textsf{T}-shaped class.

Dancing at one of the tees depends on the points adjacent to the column divider
to the right of the corner and the row divider above the corner,
whereas
dancing at the other tee depends on the points adjacent to the column divider
to the left of the corner and points adjacent to the row divider below the corner.
So these
are asymptotically independent.
Thus, from the analysis for $L_3$, and by Observation~\ref{obsIndepDancing}, we have
\[
\kappa(X_4)
\eq \kappa(L_3)^2
\eq \left( \frac{\lambda\,(1-\lambda)}{(1-(c-1)\lambda)\, (1-(r-1)\lambda)} \right)^{\!2}.
\]

\subsection{Comparison of corner types}\label{sectCornerComparison}

We can compare corner types to see which leads to the larger asymptotic growth.
The correction factors induce a partial order on the corner types, in which
corner type $\tau_1$ dominates corner type $\tau_2$
if $\kappa(\tau_1)>\kappa(\tau_2)$ for all values of $r$ and~$c$.
If this is the case, then
a class with smaller asymptotic size is the result whenever
a corner of type $\tau_1$ is replaced with one of type $\tau_2$ without changing the dimensions or the number of non-corner peaks.

The Hasse diagram of the corner type poset is shown in
Figure~\ref{figPoset}.
For a dashed edge,
strict ordering is only satisfied  under the condition specified, with the correction factors being equal otherwise.
For example, if $r=1$, then $\kappa(T_2)=\kappa(L_4)$.

\begin{figure}[ht]
  \centering
  \begin{tikzpicture}[scale=0.225]
    \draw[thick] (0,35)--(0,25);
    \draw[thick,dashed] (-10,25)--(0,30)--(10,25);
    \draw[thick] (-10,25)--(-5,20);
    \draw[thick] (10,25)--(5,20);
    \draw[thick] (-12.5,12.5)--(0,25)--(12.5,12.5);
    \draw[thick,dashed] (-5,20)--(0,15)--(5,20);
    \draw[thick] (0,15)--(-5,10)--(-5,5)--(5,10);
    \draw[thick] (0,15)--(5,10)--(5,5)--(-5,10);
    \draw[thick] (-12.5,12.5)--(0,0)--(12.5,12.5);
    \node at (-6,29-.3) {\footnotesize$c>1$};
    \node at (6,29-.3) {\footnotesize$r>1$};
    \node at (-4-.2,16+.8) {\footnotesize$r>1$};
    \node at (4+.2,16+.8) {\footnotesize$c>1$};
    \draw[radius=1.25,fill=white] (0,35) circle;        \node at (0,35) {\footnotesize$L_0$};
    \draw[radius=1.25,fill=white] (0,30) circle;        \node at (0,30) {\footnotesize$L_7$};
    \draw[radius=1.25,fill=white] (-10,25) circle;      \node at (-10,25) {\footnotesize$\rot{L_1}$};
    \draw[radius=1.25,fill=white] (0,25) circle;        \node at (0,25) {\footnotesize$X_0$};
    \draw[radius=1.25,fill=white] (10,25) circle;       \node at (10,25) {\footnotesize$L_1$};
    \draw[radius=1.25,fill=white] (-5,20) circle;       \node at (-5,20) {\footnotesize$T_2$};
    \draw[radius=1.25,fill=white] (5,20) circle;        \node at (5,20) {\footnotesize$\rot{T_2}$};
    \draw[radius=1.25,fill=white] (0,15) circle;        \node at (0,15) {\footnotesize$L_4$};
    \draw[radius=1.25,fill=white] (-12.5,12.5) circle;  \node at (-12.5,12.5) {\footnotesize$T_5$};
    \draw[radius=1.25,fill=white] (12.5,12.5) circle;   \node at (12.5,12.5) {\footnotesize$\rot{T_5}$};
    \draw[radius=1.25,fill=white] (-5,10) circle;       \node at (-5,10) {\footnotesize$L_3$};
    \draw[radius=1.25,fill=white] (5,10) circle;        \node at (5,10) {\footnotesize$X_7$};
    \draw[radius=1.25,fill=white] (-5,5) circle;        \node at (-5,5) {\footnotesize$T_4$};
    \draw[radius=1.25,fill=white] (5,5) circle;         \node at (5,5) {\footnotesize$\rot{T_4}$};
    \draw[radius=1.25,fill=white] (0,0) circle;         \node at (0,0) {\footnotesize$X_4$};
  \end{tikzpicture}
  \caption{The 
  poset of corner types}\label{figPoset}
\end{figure}

\subsection{Constrained gridded permutations}\label{sectLTXConstrained}

Recall that, if $\Grid(M)$ is a connected one-corner class, then an $M$\!-gridded permutation $\sigma^\#$ is $M$\!-constrained if
\begin{itemize}
  \item[(a)] every $M$\!-gridding of its underlying permutation $\sigma$ is the result of zero or more points of $\sigma^\#$ dancing at a peak or diagonally or through a tee,
  and
  \item[(b)] in every $M$\!-gridding of $\sigma$, each non-blank cell contains at least two points.
\end{itemize}

Suppose a gridded permutation satisfies Part~(a) of this definition.
Then, in order to satisfy Part~(b),
it is sufficient that, in each non-blank cell, there are at least two points that can't dance,
since these points will be in the same cell in all griddings of the underlying permutation.

Non-corner peak points can always dance.
However, the set of points which can dance at the corner is determined by the position of the controllers.

For peak dancing ($L_1$) and diagonal dancing ($L_7$), each controller is an \emph{extremal} (lowest, highest, leftmost or rightmost) point in one of the cells.
For diagonal dancing, the points which can dance are those that are in one of the cells adjacent to the corner and lie between the relevant controller and cell divider.

For tee dancing ($L_3$), each controller is either an extremal point in a cell, or else is the second lowest or second highest point in the corner cell.
As with diagonal dancing, points between a controller and the corresponding cell divider can dance.
However, there may also be a single point that is not between the controller and cell divider that can dance.
In addition, one or both of the extremal points in the corner cell may also be able to dance.

\begin{figure}[ht]
  \centering
  \begin{tikzpicture}[scale=0.19]
  \plotgriddedperm{16}{1,15,13,2,11,3,4,5,6,7,9,10,14,16,12,8}{8,14}{4}
  \circpt{6}{3} \circpt{7}{4} \circpt{8}{5} \circpt{9}{6} \circpt{10}{7} \circpt{11}{9} \circpt{14}{16}
  \end{tikzpicture}
  \caption{A constrained \!\gctwo{3}{-1,1,-1}{1}\!-gridded permutation;
  the circled points can dance}\label{figLTXConstrained}
\end{figure}

See Figure~\ref{figLTXConstrained} for an example of a gridded permutation in a class with $L_3$ corner type and a non-corner peak:
in each non-blank cell there are at least two points that can't dance.
Note that we may need four points in a cell above a controller in the main row to guarantee two points that can't dance,
since one point just above the controller may be able to dance, and the highest point may be a peak point that can dance.

In this context, we make the following definition.
Given a gridded permutation,
if $C_1$ and $C_2$ are two distinct cells in the same row, then they are \emph{interlocked} if
\begin{itemize}
  \item $C_1$ contains at least four points above the second lowest point in $C_2$,
  \item $C_2$ contains at least four points above the second lowest point in $C_1$,
  \item $C_1$ contains at least four points below the second highest point in $C_2$, and
  \item $C_2$ contains at least four points below the second highest point in $C_1$.
\end{itemize}
Similarly, if $C_1$ and $C_2$ are two distinct cells in the same column, then they are interlocked if
\begin{itemize}
  \item $C_1$ contains at least four points to the right of the second leftmost point in $C_2$,
  \item $C_2$ contains at least four points to the right of the second leftmost point in $C_1$,
  \item $C_1$ contains at least four points to the left of the second rightmost point in $C_2$, and
  \item $C_2$ contains at least four points to the left of the second rightmost point in $C_1$.
\end{itemize}
Note that if two adjacent cells are interlocked, then their contents together don't form an increasing or decreasing sequence.

The following proposition gives sufficient conditions for a gridded permutation in a connected one-corner class to be constrained.

\begin{prop}\label{propLTXConstrained}
  Suppose $\Grid(M)$ is a connected one-corner class and $\sigma^\#\in\Gridhash(M)$ is such that
  each pair of cells in the main row is interlocked and
  each pair of cells in the main column is interlocked.
  Then $\sigma^\#$ is $M$\!-constrained.
\end{prop}
\begin{proof}
Suppose we remove the row dividers from $\sigma^\#$, merging the cells in the main row into a single ``super cell''.
This forms a horizontal skinny gridded permutation in which the contents of one cell are not monotone.
Let us denote this~$\sigma^\|$.
Similarly, if we remove the column dividers, the result is a vertical skinny gridded permutation $\sigma^=$ with one non-monotone cell.
Because each pair of cells in the main row and each pair of cells in the main column is interlocked, there is no pair of adjacent cells
in either $\sigma^\|$ or $\sigma^=$
whose contents together form an increasing or decreasing sequence.

Thus, as in the proof of Proposition~\ref{propSkinnyConstrained} for skinny classes,
in any $M$\!-gridding of $\sigma$, by the monotonicity constraints, there must be
a column divider between each pair of adjacent monotone cells of $\sigma^\|$ that have the same orientation,
and also a column divider adjacent to each peak point of~$\sigma^\|$ (where peaks are formed from two monotone cells).
Similarly, there must be
a row divider between each pair of adjacent monotone cells of $\sigma^=$ that have the same orientation,
and also a row divider adjacent to each peak point of~$\sigma^=$.
Thus, in any $M$\!-gridding of $\sigma$, there must be a divider between each pair of adjacent non-corner cells of $\sigma^\#$ that have the same orientation
and a divider adjacent to each non-corner peak point.
So the only dancing possible across these dividers is by non-corner peak points.

What about the dividers adjacent to the corner cell?
For this, we require a case analysis of each corner type.
But this is exactly what is presented in Sections~\ref{sectPeakCorners}
to~\ref{sectTeeCorners}.
The definition of interlocking was chosen precisely so that the interlocking of cells in $\sigma^\#$ guarantees the existence of the controllers and that in each non-blank cell there are at least two points that can't dance.

By the interlocking of cells in $\sigma^\#$, every cell contains at least four points.
Each controller is either the first or second point in a cell, when considered in the appropriate direction.
Thus, in each case, the points which control the dancing at the corner are present, and the only possibilities for dancing are those described above.
So the possible griddings are restricted to those that result from zero or more points of $\sigma^\#$ dancing at a peak or diagonally or through a tee

The only points in a cell that may be able to dance are those that precede a controller and the one that immediately follows it.
The interlocking of cells in $\sigma^\#$ guarantees that there are at least four points in each cell following any controller, of which at most two (the first and last) can dance.
Thus each cell of any $M$\!-gridding of $\sigma$ contains at least two points.
\end{proof}

Finally, we prove that almost all gridded permutations in a connected one-corner class are constrained.

\begin{prop}\label{propLTXConstrainedAreGeneric}
If $\Grid(M)$ is a connected one-corner class, then almost all $M$\!-gridded permutations are $M$\!-constrained.
That is, if $\bshn$ is drawn uniformly at random from $\Gridhash_n(M)$, then
\[
\liminfty \prob{\text{$\bshn$ is $M$\!-constrained}} \eq 1 .
\]
\end{prop}
\begin{proof}
It is sufficient to prove that almost all $M$\!-gridded permutations satisfy the conditions of Proposition~\ref{propLTXConstrained}.
By Theorem~\ref{thmConnectedDistrib}, we know that the number of points in each non-blank cell of almost all $n$-point $M$\!-gridded permutations grows (linearly) with~$n$.

Suppose $C_1$ and $C_2$ are two cells in the same row, containing $\alpha n$ and $\beta n$ points respectively.
Then the probability that $C_1$ contains exactly $k$ points above the second lowest point in $C_2$ equals
\[
(\alpha n-k)\binom{\beta n+k-2}{k} \bigg/ \binom{(\alpha+\beta)n}{\alpha n}.
\]
For fixed $k$, this tends to zero as $n$ grows, since the numerator is polynomial in $n$, whereas the denominator grows exponentially.

Thus, given any pair of cells in same row or column, the probability that they are interlinked in an $n$-point $M$\!-gridded permutation converges to 1 as $n$ tends to infinity.
\end{proof}

\subsection{Beyond connected one-corner classes}\label{sectBeyondLTX}

We conclude with some brief notes on extending our approach beyond \textsf{L}, \textsf{T} and \textsf{X}-shaped classes.
Firstly, for any connected class, the asymptotic distribution of points between the cells in a typical gridded permutation can be established using the method presented in Section~\ref{sectDistribution}.
Secondly, for any connected class whose cell graph is either acyclic or unicyclic,
the approach of Section~\ref{sectLTXGridded} using generating functions can be extended to give
the asymptotic number of gridded permutations. For further details, see~\cite[Theorems~4.3 and~4.5]{BevanThesis}.

\begin{figure}[ht]
  \centering
  \gcfive{6}{0,0,0,1,1,-1}{0,0,0,1}{0,0,0,1}{0,0,0,-1}{1,1,-1,1}
  \qquad\quad
  \gcfive{7}{-1,1,-1}{0,1}{0,1,0,-1,1,1,-1}{0,1}{0,-1}
  \qquad\quad
  \gcfive{8}{0,0,0,-1}{-1,1,0,1,-1,0,1}{0,0,0,-1}{0,0,-1,-1,0,1,0,1}{0,0,0,-1}
  \qquad\quad
  \gcfive{7}{1,-1,1}{0,0,-1}{0,0,-1,-1,1,-1}{0,0,0,0,0,1}{0,0,0,0,0,-1,-1}
  \caption{Three two-corner classes, and a four-corner snake, with no adjacent corners}\label{figTwoCornersNotAdjacent}
\end{figure}

Connected two-corner classes with no adjacent corner cells, such as the three shown in Figure~\ref{figTwoCornersNotAdjacent}, don't support any new possibilities for dancing, and are thus directly amenable to the analysis presented above.
If corner orientations are restricted to northeast and southwest, then we have what we call a \emph{snake}, such as the class at the right of Figure~\ref{figTwoCornersNotAdjacent}.
The asymptotic enumeration of snakes with no adjacent corners also requires no additional analysis.

\begin{figure}[ht]
  \centering
  \gctwo{3}{-1,-1}{1,0,-1}
  \qquad\quad
  \gctwo{5}{0,1,1,0,1}{1,1,0,1}
  \qquad\quad
  \gctwo{7}{0,1,0,0,1,-1}{1,-1,-1,1,0,0,1}
  \caption{Semi-skinny classes, with two adjacent corners}\label{figSemiSkinny}
\end{figure}

However, in general, connected two-corner classes with corners in adjacent cells, such as the \emph{semi-skinny} (two row) classes shown in Figure~\ref{figSemiSkinny}, may exhibit \emph{non-corner diagonals}.
Dancing at non-corner diagonals interacts with the dancing at the corners, so additional analysis would be required.
With three or more corners, the general situation quickly gets more complicated.
With four or more corners, the cell graph may have cycles.
See~\cite{BBRUnicyclic} for a detailed investigation of the structural complexity of unicyclic classes.

In general, \emph{disconnected} classes are not directly amenable to our approach, because typical large gridded permutations in such classes may not be constrained.
For example, in \!\gctwo3{0,0,-1}{1,1}\!, the asymptotic expected number of points in the non-blank cell at the upper right is finite.
Indeed, the asymptotic probability of that cell being empty or containing just a single point is positive.
Specifically, in this class there are $2^{n-k}$ gridded $n$-permutations with $k$ points in the upper right cell (for $k=0,\ldots,n$), and therefore $2^{n+1}-1$ gridded $n$-permutations overall.
So the asymptotic probability of there being $k$ points in the upper right cell is $2^{-(k+1)}$,
and the expected number of points in the upper right cell asymptotically equals~1.
Further analysis would be required to handle this sort of situation.

\acknowledgements
We are grateful to an anonymous referee for a thorough review of an earlier version of this paper, resulting in significant improvements to its presentation.

\bibliographystyle{plain}
\bibliography{GridClasses-DMTCS}

\end{document}